\documentclass[reqno]{amsart}  
\usepackage{graphicx}
\usepackage[section]{algorithm}
\usepackage{algorithmic} 
\usepackage{amssymb, amsmath}

\usepackage{helvet}
\usepackage{courier}
\usepackage{type1cm}

\numberwithin{equation}{section}

\newcommand{\BlackBoxes}{\global\overfullrule5pt}

\BlackBoxes

\newcommand{\weakD}{\stackrel{\lower0.2ex\hbox{$\scriptscriptstyle   
             \mathbf{D}[0,1] $}}{\Rightarrow}}
\newcommand{\weak}{\stackrel{\lower0.2ex\hbox{$\scriptscriptstyle
                    \it{w} $}}{\rightarrow}}
\newcommand{\tauconv}{\stackrel{\lower0.2ex\hbox{$\scriptscriptstyle
                    \it{\tau} $}}{\rightarrow}}
\newcommand{\calX}{\mathcal{X}}

\makeatletter
\AtBeginDocument{%
\def\@serieslogo{%
\vbox to\headheight{%
\parindent\z@ \fontsize{6}{7\p@}\selectfont
August 29, 2014\endgraf
\vss}}}
\makeatother

\def\N{\mathbf{N}}

\def\R{\mathcal{R}}

\def\Prob{\mathbb{P}}
\def\E{\mathbb{E}}

\def\Eb{\bar{\mathbb{E}}}

\def\calB{\mathcal{B}}

\newcommand{\edfis}{\tilde{\mathbf{F}}}
\newcommand{\empd}{\mathbf{F}}
\newcommand{\M}{\mathcal{M}}
\newcommand{\re}{\mathcal{H}}
\newcommand{\Var}{\textrm{Var}}

\newcommand{\calF}{\mathcal{F}}

\theoremstyle{theorem}
\newtheorem{theorem}{Theorem}[section]
\newtheorem{lemma}[theorem]{Lemma}
\newtheorem{proposition}[theorem]{Proposition}

\theoremstyle{definition}

\newtheorem{example}[theorem]{Example}

\newtheorem{remark}[theorem]{Remark}

\begin{document}   

\title[Empirical measures in importance sampling]{Large deviations for weighted empirical measures
   arising in importance sampling}        

% use optional labels to link authors explicitly to addresses:   

\author[H.~Hult]{Henrik Hult$^{\dagger}$}  
\address[H.~Hult]{Department of Mathematics, KTH, 100 44 Stockholm, Sweden} 
\email{{hult@kth.se}}  

\author[P.~Nyquist]{Pierre Nyquist}
\address[P.~Nyquist]{Department of Mathematics, KTH, 100 44, Stockholm, Sweden}
\email{{pierren@kth.se}}  
\copyrightinfo{the Authors}  

\thanks{$^{\dagger}$ Support from the G\"oran Gustafsson Foundation is gratefully acknowledged.}   
%\thanks{The authors want to thank...}   

%    author one information
%\author{Henrik Hult$^{\dagger}$}
%\address[H.~Hult]{Department of Mathematics}
%\curraddr{KTH, 100 44, Stockholm, Sweden}
%\email{hult@kth.se}
%\thanks{$^{\dagger}$Support from the G\''{o}ran Gustafsson Foundation is gratefully acknowledged}

%    author two information
%\author{Pierre Nyquist}
%\address[P.~Nyquist]{Department of Mathematics}
%\curraddr{KTH, 100 44 Stockholm, Sweden}
%\email{pierren@kth.se}
%\thanks{}

%    For articles to be published after 1 January 2010, you may use
%    the following version:
%\subjclass[2010]{Primary }


\begin{abstract} 
In this paper the efficiency of an importance sampling algorithm is studied by means of large deviations for the associated weighted empirical measure. The main result, stated as a Laplace principle for these weighted empirical measures, can be viewed as an extension of Sanov's theorem. The main theorem is used to quantify the performance of an importance sampling algorithm over a collection of subsets of a given target set as well as quantile estimates. The analysis yields an estimate of the sample size needed to reach a desired precision and of the reduction in cost compared to standard Monte Carlo.
%Importance sampling is a popular method for efficient computation of various properties of a distribution such as probabilities, expectations, quantiles, etc. The output of an importance sampling algorithm can be represented as a weighted empirical measure, where the weights are given by the likelihood ratio between the original distribution and the sampling distribution. 
%In this paper the efficiency of an importance sampling algorithm is studied by means of large deviations for the weighted empirical measure. The main result, which is stated as a Laplace principle for the weighted empirical measure arising in importance sampling, can be viewed as a weighted version of Sanov's theorem. The main theorem is applied to quantify the performance of an importance sampling algorithm over a collection of subsets of a given target set as well as quantile estimates. The analysis yields an estimate of the sample size needed to reach a desired precision as well as of the reduction in cost for importance sampling compared to standard Monte Carlo.
\end{abstract}

%\begin{keyword}   
%Regular variation \sep extreme values \sep functional limit theorem \sep   
%  Markov processes   
% keywords here, in the form: keyword \sep keyword   

% PACS codes here, in the form: \PACS code \sep code   
%\MSC  60F17 \sep 60G17 (primary) \sep 60G07 \sep 60G70 (secondary)   
%\end{keyword}   
%\end{frontmatter}   

% main text   
\maketitle   
%%%%%%%%%%%%%%%%%%%%%%%%%%%%%%%%%%%%%%%%%%%%%%
\section{Introduction}
Computational methods based on stochastic simulation - the collective term for simulating a physical system, involving random coefficients, on a computer - are fundamental in a wide range of applications, including the natural sciences, computer science, economics and finance, operations research, engineering sciences such as power grids, reliability, solid mechanics, etc. As the complexity of the mathematical models used in these subjects has increased, so has the need for fast and reliable computational methods. The aim of this paper is to introduce a new approach, based on the theory of large deviations for empirical measures, for analyzing the efficiency of stochastic simulation methods. In the present paper the emphasis is on importance sampling algorithms and the results can be viewed as a first step towards a new method for analyzing efficiency that can complement the existing variance-based approach. 
%Although the focus is on importance sampling, the general idea is applicable to other simulation methods as well.

The basic idea in stochastic simulation is to generate a population of particles that moves randomly according to the laws of the physical system. Each particle carries an individual weight, which may be updated during the simulation, and quantities of the underlying system are computed by averaging the particles' weights depending on their position. The canonical example is (standard) Monte Carlo simulation where all particles are independent and statistically identical, their weights being constant and equal.

Although the standard Monte Carlo method is widely used it is by no means universally applicable. One reason is that particles may wander off to irrelevant parts of the state space, leaving only a small fraction of relevant particles that contribute to the computational task at hand. Therefore, standard Monte Carlo may come with a computational cost that is too high for practical purposes. To remedy this, a control mechanism is needed that forces the particles to move to the relevant part of the state space, thereby reducing the computational cost. Such a control mechanism may come in different forms depending on the type of algorithm under consideration. In importance sampling, see e.g. \cite{asmussen2007}, the control is the choice of sampling dynamics. In splitting algorithms \cite{DeanDupuis2009} and genealogical particle systems \cite{DelMoralGarnier, DupuisCai} the control mechanism come, roughly speaking, in the form of a birth/death mechanism, which controls that important particles give birth to new particles and irrelevant particles are killed. For more on these and other related methods see \cite{asmussen2007, RubinoTuffin} and the references therein.

Much of the theoretical work on the efficiency of stochastic simulation algorithms in general, and importance sampling algorithms in particular, is based on analyzing the variance of the resulting estimator. As long as an estimator is unbiased, variance is indeed the canonical measure of efficiency. However, when one considers situations in which the estimator is no longer unbiased, for example non-linear functionals of the underlying distribution, variance is no longer the obvious choice for determining performance and additional measures of efficiency can be useful.

Monte Carlo estimators are of the plug-in type where one first constructs an empirical measure associated with the distribution of interest and then plug this into whatever functional one is trying to estimate. A large deviation analysis for the underlying empirical measure will thus be useful in understanding the properties, such as rate of convergence, of the estimator; see \cite{TrashorrasWintenberger} for a recent example regarding bootstrap estimators. Therefore, much like for many other statistical estimators, a natural progression in the study of importance sampling algorithms and their performance is to consider the large deviation properties of the underlying empirical measure.

We propose in this paper that large deviation theory for the empirical measures arising in importance sampling can be used to define new performance criteria, based on the associated rate function, to complement the variance analysis. In addition to deriving the relevant large deviation result, we introduce one such criterion, which we refer to as evaluating performance over subsets. The intuition is that if a simulation algorithm is to approximate the original system in a certain region of the state space, then it should do so equally well over all subsets of that region that are not, in a sense, too small. 

To avoid confusion it should be pointed out that techniques from sample-path large deviations have been studied thoroughly, to control the rate at which the variance decays, in the context of designing efficient rare-event simulation algorithms; see for example \cite{DupuisWang,DupuisWangSubsol} for results on the so-called subsolution approach. Our objective is fundamentally different. We want to complement the variance by the rate function of a large deviation principle as the number of particles increases. For our purposes the appropriate framework is large deviations for empirical measures, in the spirit of Sanov's theorem \cite[Theorem 2.2.1]{Dupuis97}, 
rather than sample-path large deviations. The suggested approach is applicable to a wide range of simulation problems and is not intended exclusively for problems in rare-event simulation. Furthermore, the result is potentially useful for theoretical comparison of the performance of algorithms of different character, based on, say, importance sampling and interacting particle systems, by comparing the associated rate functions. 

The contribution of the paper is two-fold. First, we prove a Laplace principle, in the space of finite measures equipped with the $\tau$-topology, for the weighted empirical measure arising in importance sampling. This can be viewed as an extension of Sanov's theorem and under certain conditions the result follows from the latter and the contraction principle. Our proof of the general version, stated in Theorem \ref{thm:Laplace}, is based on the weak convergence approach to large deviations. It leads to a method for theoretical quantification of performance for importance sampling algorithms for non-linear functionals of the underlying distribution such as quantiles and other risk measures.

Second, we study the associated rate function and introduce a performance criteria for quantifying the performance of an importance sampling algorithm over a collection of subsets of a given target set as well as quantile estimates. A compelling feature of the large deviation approach is that it provides explicit estimates of the sample size needed to reach a desired accuracy as well as of the reduction in computational cost for importance sampling compared to standard Monte Carlo. 

The outline of the paper is as follows. In Section \ref{sec:background} an introduction to importance sampling is presented along with background on variance based efficiency analysis and large deviations for empirical measures. The main theoretical result, a Laplace principle for the weighted empirical measure of an importance sampling algorithm, is presented in Section \ref{sec:Main_results}. Applications of the large deviation result to efficiency analysis and design of importance sampling algorithms are given in Section \ref{sec:applications}. Most of the examples are well-studied elsewhere and are mainly intended to demonstrate the efficiency analysis using the rate function in contrast to the standard variance analysis. Section \ref{sec:proofs} contains the proof of the Laplace principle in Section \ref{sec:Main_results}. 

%%%%%%%%%%%%%%%%%%%%%%%%%%%%%%%%%%%%%%%%%%%%%%%%%%%
\section{Background and motivation} \label{sec:background}
Let $\calX$ be a complete separable metric space equipped with its Borel $\sigma$-field $\calB(\calX)$ and let $(\Omega, \calF, \Prob)$ be a probability space. Unless otherwise stated subsets of $\calX$ under consideration are always assumed in $\calB(\calX)$. Consider a random variable $X: \Omega \to \calX$ with distribution $F$. Denote by $\M_{1}$ the space of probability measures on $\calX$. The objective is to approximate $F$ in a given region $A \in \calB(\calX)$ or to compute $\Phi (F)$, where $\Phi : \M_{1} \to \R$ is a given real-valued functional. Examples of functionals that may be of interest are expectations, 
\begin{align*}
	\Phi _f(F) = \int f dF, \text{ for some } f: \calX \to \R, 
\end{align*}
and, when $X$ is real valued, quantiles
\begin{align*}
	\Phi_q (F) = F^{-1}(q) = \inf \{x: F((x,\infty)) \leq q \}, \;q \in (0,1),
\end{align*}
and $L$-statistics, 
\begin{align*}
	\Phi (F) = \int _{0} ^1 \phi (q) F^{-1}(q)dq. 
\end{align*}

Only in exceptional cases is explicit computation of such functionals possible. When explicit computations are not possible a viable alternative is simulation.
The standard method of simulation is Monte Carlo, in which the empirical measure 
\begin{align*}
\empd _n = \frac{1}{n} \sum_{k=1}^{n} \delta _{X_k},
\end{align*}
 is constructed from an independent sample $X_1,...,X_n$ from $F$. Here $\delta _x$ denotes a unit point mass at $x$. The quantity $\Phi (F)$ is estimated by the plug-in estimator $\Phi (\empd _n)$. Roughly speaking, if $\empd _n$ is a good approximation of $F$ in a region that largely determines $\Phi (F)$, then $\Phi (\empd _n)$ is likely to provide a good estimate.

\subsection{Large deviation analysis for quantifying the performance of Monte Carlo algorithms}
%
%Let us try to quantify how efficient the plug-in estimator is by means of large deviations for the associated empirical measure. 
%
The precision of the estimator $\Phi(\empd_{n})$ will of course be affected by the sample size $n$. By the law of large numbers for empirical measures $\empd _n$ converges weakly to $F$ with probability one,  as $n \to \infty$, and an increased sample size $n$ will thus improve the accuracy in this sense. If  $\Phi (\empd _n)$ is an unbiased estimate of $\Phi (F)$, then the sample size required to reach a desired precision is well-captured by $\Var (\Phi(\empd _n))$ and the performance of an estimator can analyzed in terms of the variance; see \cite{asmussen2007} for different variance-based performance criteria. However, in the general case it may be that looking solely at the variance of the estimator is insufficient. One might use the mean-squared error $\E [(\Phi (\empd _n) - \Phi (F) )^2]$ instead but just as the variance this is not necessarily easy to analyze for a general $\Phi$.

An alternative way of quantifying the efficiency of the Monte Carlo estimator is through the theory of large deviations.
% A sequence of random variables $\{X_n\}$ taking values in a topological space $\mathcal{U}$ is said to satisfy a \emph{large deviation principle} (LDP) on $\mathcal{U}$ with rate function $I$ if it satisfies
%$$ \limsup _n \frac{1}{n} \log \Prob(X_n \in F) \leq - \inf _{u \in F} I(u), $$
%and
%$$ \liminf _n \frac{1}{n} \log \Prob(X_n \in G) \geq - \inf _{u \in G} I(u), $$
%for each closed $F \subset \mathcal{U}$ and each open $G \subset \mathcal{U}$, respectively. See e.g.\ \cite[p.\ 5]{Dembo98}. For notational purposes, let $I(F):= \inf _{u \in F} I(u)$. 
To illustrate the point, let us consider the example of computing the expectation $F(f) = \int f dF$ for some $f:\calX \mapsto \R$ by Monte Carlo simulation. An estimate of $F(f)$ is then given by
\begin{equation} 
\label{eq:MC_mean}
\empd _n (f) = \frac{1}{n} \sum _{i=1} ^n f(X_i),
\end{equation}
where the $X_i$'s are independent with common distribution $F$. 

Cram\'er's theorem  states that if $\E[\exp\{\theta f(X)\}] <\infty$ for $\theta$ in a neighborhood of the origin, then $\empd_{n}(f)$ satisfies the large deviation principle:
\begin{align*}
	- I(A^{\circ}) &\leq \liminf_{n \to \infty} \frac{1}{n} \log \Prob(\empd_{n}(f) \in A^{\circ}) \\
	& \leq \limsup_{n \to \infty} \frac{1}{n} \log \Prob(\empd_{n}(f) \in \bar{A}) \leq - I(\bar{A}),
\end{align*}
for all Borel sets $A$, where $\bar{A}$ and $A^{\circ}$ denote the closure and interior of $A$, respectively, $I(A) = \inf_{x \in A}I(x)$, and the rate function $I$ is given by
\begin{align*}
	I(x) = \sup_{\theta} \{ \theta x - \kappa(\theta) \}, 
\end{align*}
where $\kappa(\theta) = \log \E[\exp\{\theta f(X)\}]$ is the logarithmic moment generating function \cite[p.\ 26]{Dembo98}.

Suppose for the sake of illustration that, with probability at least $1-\delta$, a relative precision $\epsilon$ is desired in the estimate. That is, the sample size $n$ must be selected sufficiently high that
$$ \Prob (|\empd _n(f) - F(f)|\geq \epsilon F(f)) \leq \delta.$$
Let $A_{\epsilon}$ denote the complement of the open ball of radius $\epsilon F(f)$ centred at $F(f)$, 
\begin{align*}
	A _{\epsilon} = \{ x \in \R : | x - F(f)| \geq \epsilon F(f) \}.
\end{align*}
By Cram\'{e}r's theorem,
\begin{align*}
%\label{eq:Cramer}
	%\limsup _n \frac{1}{n}\log \Prob (|\empd _n(f) - F(f)|\geq \epsilon F(f)) = 
	\limsup _n \frac{1}{n}\log \Prob (\empd _n(f) \in A_{\epsilon}) \leq - I(A _{\epsilon}), 
\end{align*}
which implies that for large $n$, at least approximately, 
\begin{equation*}
	\Prob (\empd _n(f) \in A_{\epsilon}) \lessapprox e^{-n I(A _{\epsilon})}.
\end{equation*}
An upper bound $\delta$ on the error probability corresponds to 
\begin{equation}\label{eq:nCramer}
	n \gtrapprox  \frac{1}{I(A_{\epsilon})}(-\log \delta).
\end{equation}
Roughly speaking the sample size must be proportional to the reciprocal of the rate for the error probability in order to obtain the desired performance. 

Let us consider a more general example where we are interested in approximating the true distribution $F$ over a region $A \in \calB(\calX)$. Again the empirical measure $\empd_{n}$ resulting from Monte Carlo simulation, restricted to $A$, provides a viable approximation. In this context the large deviation principle for the sequence of empirical measures can be applied to quantify the performance of the Monte Carlo algorithm. Sanov's theorem \cite[Theorem 2.2.1]{Dupuis97} states that the sequence $\{ \empd_{n} \}$ satisfies the large deviation principle on $\M_{1}$ with rate function given by the relative entropy $\re(\cdot \mid F)$, defined by
$$ \re (G \mid F) = \int \log \frac{dG}{dF}  dG, \quad G \in \M_{1}. $$
More precisely, for $A \in \calB(\M_{1})$
\begin{align*}
	- I(A^{\circ}) & \leq \liminf_{n \to \infty} \frac{1}{n} \log \Prob(\empd_{n} \in A^{\circ}) \\
	&\leq \limsup_{n \to \infty} \frac{1}{n} \log \Prob(\empd_{n} \in \bar{A}) \leq - I(\bar{A}).
\end{align*} 
Taking $A = U_{G}$ to be a small neighborhood of $G$ one can interpret, for $n$ large, 
 $\exp\{-n I(G)\}$, where $I(G) = \re(G \mid F)$, as an approximation of the probability that the empirical measure looks like a typical sample from $G$. If it is undesirable that the empirical measure looks like a typical sample from a specific distribution $G$, say the objective is to have
\begin{align*}
	\Prob(\empd_{n} \in U_{G}) \leq \delta, 
\end{align*}
then, by the same reasoning as above, the sample size must be selected sufficiently large that
\begin{equation}\label{eq:nG}
	n \gtrapprox  \frac{1}{I(G)}(-\log \delta).
\end{equation}
Notice that taking $A = A _{\epsilon}$ as defined above in Sanov's theorem one can recover \eqref{eq:nCramer}. 

In the context of a general functional $\Phi$ Sanov's theorem may also be applied. Either one would try to protect against a certain undesirable shape $G$ of the empirical measure by selecting $n$ as in \eqref{eq:nG} or by looking at error probabilities corresponding to sets of the form
$$ A_{\epsilon} = \{ G \in \mathcal{M} _1: |\Phi(G) - \Phi(F)| \geq \epsilon \Phi (F) \}. $$
It may be pointed out that when considering a set such as $A_{\epsilon}$ the rate is given by 
\begin{align*}
	I(A_{\epsilon}) = \inf_{G \in A_{\epsilon}} I(G).
\end{align*}
If the infimum is attained at $G_{\epsilon}$, then limiting the probability of $A_{\epsilon}$ corresponds precisely to protecting against $G_{\epsilon}$. 
%%%%%%%%%%%%%%%%%%%%%%%%%%%%%%%%%%%%%%%%%%%%%%%%%
\subsection{Importance sampling}
The basic idea of importance sampling is to draw samples from a sampling distribution that is more likely to generate samples from the desired region, thus hopefully improving accuracy compared to standard Monte Carlo. Suppose that the goal is to evaluate $\Phi(F)$ for some functional $\Phi$. For simplicity, start by considering the case of an expectation, $\Phi(F) = F(f) = \int f(x) F(dx)$ for an $F$-integrable, non-negative function $f: \calX \to \R _+ $. Let $\tilde{F}$ be the chosen sampling distribution. For $\tilde{F}$ to be a feasible sampling distribution it must hold that $F \ll \tilde{F}$ on the support of $f$. If this holds, then the Radon-Nikodym derivative $dF/d\tilde F$ exists on $\{f > 0\} $ and it is possible to define the weight function 
\begin{equation}
\label{eq:weight_func}
w(x) = \frac{dF}{d\tilde{F}}(x) I\{ f(x) > 0\}.
\end{equation}
Let $\tilde X_1, \dots, \tilde X_n$ be an independent sample from $\tilde{F}$. The weighted empirical measure corresponding to the importance sampling algorithm is 
$$ \edfis ^w _n = \frac{1}{n} \sum _{k=1} ^n w(\tilde X_k)\delta _{\tilde X_k}.$$
Note that in contrast to standard Monte Carlo $\edfis _n^w$ is typically not  a probability measure.  The importance sampling estimator of $F(f)$ is the plug-in estimator
\begin{align*}
\edfis _n ^w (f) = \frac{1}{n} \sum _{i=1} ^n w(\tilde X_i)f(\tilde X_i).
\end{align*}

Let $\widetilde \Prob$ and $\widetilde \E$ denote the probability and expectation when $\tilde X_{1}, \tilde X_{2}, \dots $ are sampled from the sampling distribution $\tilde F$. If $\widetilde \E[\exp\{\theta w(\tilde X)f(\tilde X)\}] < \infty$ for $\theta$ in a neighborhood of the origin, then Cram\'er's theorem implies that $\edfis^{w}_{n}(f)$ satisfies the large deviation principle with rate function 
\begin{align*}
	I^{w}(x) &= \sup_{\theta} \{ \theta x - \kappa^{w}(\theta) \},
\end{align*}
where $\kappa ^w$ is the associated logarithmic moment generating function,
\begin{align*}
	\kappa^{w}(\theta) &= \log \widetilde \E[\exp\{\theta w(\tilde X)f(\tilde X)\}].  
\end{align*}

Similar to the Monte Carlo illustration, suppose that a relative precision $\epsilon$ is desired in the estimate with probability at least $1-\delta$. That is, the sample size $n$ must be selected sufficiently high that
$$ \widetilde \Prob (|\edfis^{w}_n(f) - F(f)|\geq \epsilon F(f)) \leq \delta.$$
Just as before Cram\'er's theorem, with $A_{\epsilon}$ denoting the complement of the open ball of radius $\epsilon F(f)$ centered at $F(f)$,  implies that the sample size must satisfy
\begin{equation*}
	n \gtrapprox  \frac{1}{I^{w}(A _{\epsilon})}(-\log \delta).
\end{equation*}
The choice of the sampling distribution enters in the rate function through the weights $w = (dF/d\tilde F)I\{f > 0\}$. The improvement over standard Monte Carlo can be quantified by comparing the rate function $I$ corresponding to Monte Carlo and the rate function $I^{w}$ corresponding to importance sampling. Furthermore, a good choice of the sampling distribution in the importance sampling algorithm is one that maximizes the rate $I^{w}(A _{ \epsilon})$.

To extend the analysis to more general functionals it is desirable to have an analogue of Sanov's theorem for the weighted empirical measures $\edfis^{w}_{n}$. In contrast to Monte Carlo one cannot expect that the weighted empirical measure $\edfis^{w}_{n}$ is a good approximation to $F$ everywhere. Rather, the sampling distribution is selected to obtain good precision in the important part of the space. If the objective is to compute $\Phi(F)$ for some functional $\Phi$, then it suffices that $\edfis^{w}_{n}$ approximates $F$ in the region that largely determines $\Phi(F)$. For this purpose a non-negative measurable function $f: \calX \to \R _+$ is introduced, called the \emph{importance function}. 

The rough statement that $\edfis _n ^w$ is close to $F$ in the important region is made precise by saying that the measure $\edfis ^{wf} _n$ is close to $F^f$ in the space $\M = \M(\calX)$ of finite measures, where 
\begin{align*}
	\edfis^{wf}_{n} = \frac{1}{n}\sum_{k=1}^{n} w(X_{k})f(X_{k})\delta_{X_{k}},
\end{align*}
and $F^{f}$ is the finite measure given by
\begin{equation*}
	F^f(g) = \int  g(x)f(x) dF(x),
\end{equation*}
for each bounded measurable $g: \calX \to \R$. In order for an efficiency analysis as the one described above to be feasible, the main objective is to establish an analogue of Sanov's theorem for the weighted empirical measures. 
%\section{Motivation and main results} 
%%%%%%%%%%%%%%%%%%%%%%%%%%%%%%%%%%%%%%%%%%%%%%%%%%
\section{A Laplace principle for weighted empirical measures}\label{sec:Main_results}
In this section the main theoretical result of this paper is stated. It is an extension of Sanov's theorem to the weighted empirical measures arising in importance sampling, stated as a Laplace principle. 

A sequence of random variables ${U_n}$ taking values in a topological space $\mathcal{U}$ is said to satisfy the \emph{Laplace principle} on $\mathcal{U}$ with rate function $I$ if, for all bounded, continuous functions $h: \mathcal{U} \to \R$, %(for $\mathcal{U}$ a Polish space, see Definition 1.2.2 in \cite{Dupuis97})
\begin{align*}
	\lim _{n \to \infty} \frac{1}{n} \log \E [e^{-nh(U_n)}] = - \inf _{u \in \mathcal{U}} \{ h(u) + I(u)\}.
\end{align*}
When $\mathcal{U}$ is a Polish space, the Laplace principle is equivalent to the sequence satisfying the large deviation principle with the same rate function \cite[Theorems 1.2.1, 1.2.3]{Dupuis97}. 
 In the case of a general topological space, the relationship between the large deviation principle and the Laplace principle is given by Varadhan's lemma and Bryc's inverse, see \cite{Dembo98} and the references therein.

Suppose that $F$ and $\tilde{F}$ are two given distributions on a complete, separable metric space $\calX$. The space $\M$ of finite measures on $\calX$ is equipped with the $\tau$-topology: $\nu_{n} \tauconv \nu$ if $\nu_{n}(g) \to \nu(g)$ for all bounded measurable $g : \calX \to \R$. To avoid the subtle measurability issues of the $\tau$-topology, see \cite{Dupuis97}, we will have to work with $\calF_{\M}$, the smallest $\sigma$-field on $\M$ with respect to which the function mappings $\nu \mapsto \int g d\nu$ are measurable. 

Let $f$ be an importance function. That is, $f$ is a non-negative $F$-integrable function characterizing the importance of different regions of $\calX$. It is assumed that $F \ll \tilde{F}$ on the support of $f$ and the weight function $w$ is defined as in 
(\ref{eq:weight_func}). Let $\M _{1}$, the space of probability measures, be equipped with the  $\tau$-topology and introduce the set 
$$\Gamma = \{ G \in \M _{1} : G(wf)  < \infty \}.$$ 
Define the mapping $\Psi$ from the subset $\Gamma \subset \M _{1} $ to $\M $ as the mapping for which $\Psi (G;\cdot)$ is the finite measure given by
\begin{align*}
	\Psi (G;g) = \int g(x)f(x)w(x) dG(x),
\end{align*}
for each bounded measurable $g: \calX \to \R$. A key observation is that $\edfis _n ^{wf} (g) = \edfis _n (wfg)$, where $\edfis _n$ is the empirical measure obtained by sampling from $\tilde F$,
\begin{align*}
\edfis _n = \frac{1}{n} \sum_{k=1}^{n} \delta_{\tilde X_{k}}.
\end{align*} 
Therefore 
$\edfis _n ^{wf} (\cdot) = \Psi(\edfis _n; \cdot)$. Note also that 
$\edfis _n$ belongs to $\Gamma$ with probability $1$.

Let $\Delta \subset \M _1$ be the set 
\begin{equation}
\label{eq:Delta}
\Delta = \{ G \in \M _1 :  \re (G \mid \tilde{F} ) < \infty \}.
\end{equation}
%\begin{definition}
%\label{def:Rate}
We are now ready to introduce the rate function. Let $I: \M \mapsto [0,\infty]$ be the function defined by
\begin{equation}
\label{eq:Rate}
 I(\nu) = \inf \{ \re(G \mid \tilde{F}): G \in \Gamma \cap \Delta, \Psi(G) = \nu \},
 \end{equation}
when such a $G$ exists and $I(\nu) = \infty$ otherwise.
%\end{definition}
Proposition \ref{prop:Rate} below states that  $I$ has sequentially compact level sets. 
Our main result is the following. 

\begin{theorem}[Laplace principle for weighted empirical measures]
\label{thm:Laplace}
Let $F$ and $\tilde{F}$ be given as above and let $f$ be an importance function. Suppose that 
\begin{enumerate}
\item[(i)] there exists a function $U: \calX \to [0,\infty]$ such that
$\int  e^{U(x)} d\tilde{F}(x) < \infty$ and $U$ has relatively compact level sets,
\item[(ii)]  $\int e^{\alpha w(x)f(x)}d\tilde{F}(x) < \infty$ for all $\alpha > 0$.
\end{enumerate}
The sequence $\{ \edfis ^{wf} _n\} $ of weighted empirical measures satisfies the 
Laplace principle on $\M$ equipped with the $\tau$-topology. For all bounded  
continuous functions $h: \M \to \R$, measurable on $(\M, \calF_{\M})$, 
\begin{equation}
\label{eq:Laplace_thm}
 \lim _{n \to \infty} \frac{1}{n} \log \widetilde\E [e ^{-nh(\edfis ^{wf} _n)}] = -\inf_{\nu \in \M} \{ h(\nu) + I(\nu)\},
\end{equation}
with the rate function $I$ given by \eqref{eq:Rate}.
\end{theorem}
The proof is given in Section \ref{sec:proofs}. 
\begin{remark}
Condition (i) is a  version of Condition 8.2.2 in \cite{Dupuis97}, adapted to the case of independent and identically distributed variables.
In the case of real-valued random variables it is a very mild assumption. Take, for example,
$$ U(x) = \max \{ \alpha \log (|x|), 0 \},$$
for some $\alpha > 0$. Then $U$ has relatively compact level sets and the condition of a finite expectation of $e^{U(X)}$ is in this case weaker than
$$ \mathbb{E} [|X|^{\alpha}] < \infty, $$
for some $\alpha > 0$. 
\end{remark}
\begin{remark}
\label{re:Laplace}
The right-hand side of  \eqref{eq:Laplace_thm} can be expressed as
$$ \inf_{\nu \in \M} \{ h(\nu) + I(\nu) \} = \inf _{G \in \Gamma \cap \Delta} \{ h(\Psi(G)) + \re(G | \tilde{F}) \}. $$
Although the expression on the left is the standard way to express the limit of the Laplace principle, the expression on the right better suits the weak convergence approach to large deviations that will be adopted throughout the proof. 
\end{remark}

Consider the special case when the function $wf$ is bounded. Then, $\Psi : \Gamma \to \M$ is continuous when both spaces are equipped with the $\tau$ topology and Theorem \ref{thm:Laplace} follows essentially from a standard application of the contraction principle, see, e.g., \cite[Theorem 1.3.2]{Dupuis97}. The main difficulty is to show that the Laplace principle holds also in the general case.
%%%%%%%%%%%%%%%%%%%%%%%%%%%%%%%%%%%%%%%%%%%%%%%%%%
\section{Applications in performance analysis} \label{sec:applications}
In this section we propose one way in which the Laplace principle of Theorem \ref{thm:Laplace} can be used to characterize the performance of an importance sampling algorithm. We outline a method for analyzing the performance over a collection of subsets of a target region $A$. If the sampling distribution is designed for a target set $A$, then one can expect to have good performance for subsets $C$ of $A$ that are not too small relative to $A$. To illustrate the performance analysis based on this criterion we consider some well-known examples. In addition to explicitly computing the rate function associated with Theorem \ref{thm:Laplace} we obtain estimates on the reduction in computational cost when comparing standard Monte Carlo and importance sampling. We also discuss briefly the rare-event limit, when the target set has small probability. The section ends with a brief discussion on performance analysis for importance sampling algorithms designed for computing the quantile of a distribution.  
%%%%%%%%%%%%%%%%%%%%%%%%%%%%%%%%%%%%%%%%%%%%%%%%%%%
\subsection{Performance over a collection of subsets}
In this section we are interested in the performance of importance sampling algorithms over a region $A \subset \calX$, reflected in the importance function $f(x) = I\{x \in A\}$. The ideal is that the weighted empirical measure $\edfis ^{wf} _{n}$ is close to $F$ on all measurable subsets of $A$. For large $n$ we can imagine that the weighted empirical measure looks like a typical sample from a measure $\nu$, which is absolutely continuous with respect to $F$. The performance of the importance sampling algorithm is good if it is likely that $\edfis^{wf}_{n}$ looks like a typical sample from some $\nu$ belonging to a set of measures  for which the likelihood ratio $d\nu/dF$ is close to $1$ on $A$. For given $\epsilon >0$ and $\delta >0$ (where $\delta = \delta' F(A)$ for some $\delta'>0$ is a reasonable choice), consider the sets
\begin{align}
	A_{\epsilon, \delta} &= \big\{ \nu \in \M : |d\nu/dF(x) -1| \geq \epsilon \text{ for } x \in \,\text{some }C \subset A \text{ with } F( C) \geq \delta \big \}, \nonumber \\
	A^+ _{\epsilon, \delta} &= \big \{\nu \in \M : d\nu/dF(x) \geq 1 + \epsilon \text{ for } x \in \,\text{some }C \subset A \text{ with } F( C) \geq \delta \big \},  \label{eq:alt1}\\ 
	A^- _{\epsilon, \delta} &= \big \{\nu \in \M: d\nu/dF(x) \leq 1 - \epsilon \text{ for } x \in \,\text{some }C \subset A \text{ with } F( C) \geq \delta \big \}.   \label{eq:alt2}
\end{align}
The rate $I(A_{\epsilon, \delta})$, with $I$ is as in Theorem \ref{thm:Laplace},  can be used to evaluate the performance of the importance sampling algorithm. The interpretation is that  $e^{-nI(A_{\epsilon, \delta})}$ is roughly the probability that $\edfis^{wf}_{n}$ provides an approximation of $F$ with relative error greater than $\epsilon$ for some subset(s) $C \subset A$ that is not too small in the sense that $F( C) \geq  \delta$. The sets $A^+ _{\epsilon, \delta}$ and $A^-_{\epsilon, \delta}$ have similar interpretations for overestimation and underestimation, respectively. The sets $A^+ _{\epsilon, \delta}$ and $A^-_{\epsilon, \delta}$ are somewhat easier to analyze than $A_{\epsilon, \delta}$, so for the sake of illustration they will be studied throughout the rest of this section. For brevity, in the examples that follows we will work exclusively with $A^{+}_{\epsilon, \delta}$. 

 If $G$ is a probability measure such that $\nu = \Psi(G)$, then $d\nu/dF = dG/d\tilde F$ on $A$ and
\begin{align*}
	I(A^+ _{\epsilon, \delta}) = \inf \Big \{ \re(G \mid \tilde F): \frac{dG}{d\tilde F}(x) \geq 1+\epsilon \text{ for } x \in \, \text{some } C \subset A, \ F( C) \geq \delta \Big \}. 
\end{align*}
Let us compute $I(A^{+}_{\epsilon, \delta})$. To start off, consider a fixed set $C \subset A$. 
\begin{lemma} \label{lemma:rateExp}
	Given $C \subset A$ it holds that
\begin{align*}
	& \inf \Big \{ \re(G \mid \tilde F):  \frac{dG}{d\tilde F}(x) \geq 1+\epsilon \text{ for } x \in C \Big \} \\ & \quad = (1+\epsilon) \tilde F( C)\log(1+\epsilon) + (1-(1+\epsilon)\tilde F( C))\log\Big(\frac{1-(1+\epsilon)\tilde F(C )}{1-\tilde F( C)}\Big),
\end{align*}
where the infimum is attained for the probability measure $G^{*}$ with 
\begin{align*}
	\frac{dG^{*}}{d\tilde F}(x) &= 1+\epsilon, \text{ for } x \in C, \\ \frac{dG^{*}}{d\tilde F}(x) &= \frac{1-(1+\epsilon)\tilde F(C )}{1-\tilde F( C)}, \text{ for } x \in C^{c}.
\end{align*}
\end{lemma}
\begin{proof}
	For any probability measure $G$ with $\frac{dG}{d\tilde F}(x) \geq 1+\epsilon$ on $C$, convexity of the function $\varphi(s) = s\log s$ and Jensen's inequality implies that
\begin{align*}
	\re(G \mid \tilde F) &= \tilde F( C) \int_{C} \varphi \biggl(\frac{dG}{d\tilde F}(x) \biggr) \frac{d \tilde F(x)}{\tilde F( C)} +  
	\tilde F( C^{c}) \int_{C^{c}} \varphi \biggl(\frac{dG}{d\tilde F}(x) \biggr) \frac{d \tilde F(x)}{\tilde F( C^{c})} \\
	& \geq \tilde F(C )\biggl( \frac{G(C )}{\tilde F( C)} \log \frac{G( C)}{\tilde F( C)}\biggr) + 
	\tilde F(C^{c} )\biggl( \frac{G(C^{c} )}{\tilde F( C^{c})} \log \frac{G( C^{c})}{\tilde F( C^{c})}\biggr)\\
	& = G( C)(\log G( C) - \log \tilde F(C )) \\
	&\quad + (1-G(C )) [\log(1-G(C )) - \log(1-\tilde F( C))].
\end{align*}
This is convex as a function of $G( C)$ and the minimizer is $G(C ) = (1+\epsilon)\tilde F(C )$.  We conclude that the lower bound
\begin{align*}
&\inf \Big \{ \re(G \mid \tilde F):  \frac{dG}{d\tilde F}(x) \geq 1+\epsilon \text{ for } x \in C \Big \} \\ & \quad \geq  (1+\epsilon) \tilde F( C)\log(1+\epsilon) + (1-(1+\epsilon)\tilde F( C))\log\Big(\frac{1-(1+\epsilon)\tilde F(C )}{1-\tilde F( C)}\Big),
\end{align*}
holds. It is straightforward to check that the lower bound is attained by $G^{*}$. This completes the proof.  
\end{proof}

Denote by $J_+$ the set function given by
\begin{align*}
	J_+(C ) = \inf _{G \in \Gamma \cap \Delta} \Big \{ \re(G \mid \tilde F): \frac{dG}{d\tilde F}(x) \geq 1 + \epsilon \text{ for } x \in C \Big \}.  
\end{align*}
The rate $I(A^{+}_{\epsilon, \delta})$ can be computed by  minimizing the function $J_+$ over the collection of feasible sets.
\begin{lemma}\label{lem:h1} $I(A^+ _{\epsilon, \delta}) = \inf\{ J_+( C) : C \subset A, F(C ) \geq \delta\}. $
\end{lemma}
\begin{proof} For any $C^* \subset A$ such that $F(C^*) \geq \delta$ we have
\begin{align*}
	I(A^+ _{\epsilon, \delta}) &= \inf \Big \{ \re(G \mid \tilde F): \frac{dG}{d\tilde F}(x) \geq 1 + \epsilon \text{ for } x \in \text{ some } C \subset A, \ F( C) \geq \delta \Big \} \\
	&\leq \inf \Big \{ \re(G \mid \tilde F):  \frac{dG}{d\tilde F}(x) \geq 1 + \epsilon \text{ for } x \in C^* \Big \} \\
	&= J_+(C^*).
\end{align*}
Taking infimum over feasible sets $C^{*}$ leads to 
$$ I(A^+ _{\epsilon, \delta}) \leq \inf \big \{ J_+ ( C^{*}) : C^{*} \subset A, F(C^{*}) \geq \delta \big \}. $$
It remains to show the reverse inequality. For every $\eta > 0$ there exists a probability measure $G^* \in \Delta \cap \Gamma$ and a corresponding set $C^* \subset A$, with 
 $F(C^*) \geq \delta$, such that
$$ I(A^+ _{\epsilon, \delta}) + \eta \geq \re(G^* | \tilde{F}) \geq J_+(C^*) \geq \inf\{ J_{+}( C): C \subset A \text{ and } F( C) \geq \delta\}.$$
Since $\eta > 0$ is arbitrary the proof is complete.
\end{proof}
 
The next result characterizes the minimizing set $C$ of Lemma \ref{lem:h1} in terms of the likelihood ratio between $F$ and $\tilde F$. 

 \begin{proposition} \label{prop:Set}
 For any $t \geq 0$, let
\begin{align*}
	C_{t} = \Big \{x \in A: \frac{dF}{d\tilde F}(x) \geq t \Big \}.
\end{align*}
For any $\delta > 0$, if there exists $\tilde t_{\delta}$ such that $F(C_{\tilde t_{\delta}}) = \delta$, then 
the infimum in Lemma \ref{lem:h1} is attained by $C_{\tilde t_{\delta}}$. That is, 
\begin{align*}
	I(A^{+}_{\epsilon, \delta}) = \inf\{ J_+( C) : C \subset A, F(C ) \geq \delta\} = J_{+}(C_{\tilde t_{\delta}}).
\end{align*}
In general, 
\begin{align*}
	\sup\{J_{+}(C_{t}): t\geq 0, F(C_{t}) \leq \delta\} \leq I(A^{+}_{\epsilon, \delta}) \leq  \inf\{J_{+}(C_{t}): t\geq 0, F(C_{t}) \geq \delta\}. 
\end{align*}
 \end{proposition}
\begin{proof} The expression for $\re(G^* \mid \tilde{F})$ in Lemma \ref{lemma:rateExp} is increasing in $\tilde{F}( C)$. Thus, minimizing $J_+(C)$ corresponds to taking infimum of $\tilde{F} (C)$. Consider the problem
$$ \inf \{ \tilde{F}(C): C \subset A, F(C) \geq \delta \}.$$
Let $\mathcal{F} _A$ be the collection of all measurable functions $\alpha: A \mapsto [0,1]$. Since indicator functions of subsets of $A$ are included in $\mathcal{F} _A$, 
\begin{align*}
& \inf \{ \tilde{F}(C): C \subset A, \ F(C) \geq \delta \} \\
& \quad \geq \inf \Big \{ \int _A \alpha (x) d \tilde{F}(x): \alpha \in \mathcal{F} _A, \int _A \alpha(x) dF(x) \geq \delta \Big \},
\end{align*}
with equality if the infimum on the right-hand side is attained for an indicator function.

Suppose that there exists $\tilde t_{\delta}$ with $F(C_{\tilde t_{\delta}}) = \delta$. 
We claim that $\alpha^* (x) = I\{ x \in C_{\tilde t_{\delta}}\}$ is the solution to the problem on the right in the last display. To see this, take an arbitrary $\alpha \in \mathcal{F} _A$ such that $\int _A \alpha(x) dF(x) \geq \delta $. 
\begin{align*}
	\int _A (\alpha(x) - \alpha^*(x)) d\tilde{F}(x) &= \int _A (\alpha(x) - \alpha^*(x))\frac{d\tilde{F}}{dF}(x)dF(x) \\
	&= \int _{C_{\tilde t_{\delta}}} (\alpha(x) - 1)\frac{d\tilde{F}}{dF}(x)dF(x) \\
	& \quad +  \int _{A \setminus C_{\tilde t_{\delta}}} \alpha(x)\frac{d\tilde{F}}{dF}(x)dF(x) \\
	&\geq \int _{C_{\tilde t_{\delta}}} (\alpha(x) - 1)\frac{1}{\tilde t_{\delta}} dF(x) +  \int _{A \setminus C_{\tilde t_{\delta}}} \alpha(x) \frac{1}{\tilde t_{\delta}} dF(x) \\ 
	&= \frac{1}{\tilde t_{\delta}} \int _A \alpha(x) dF(x) - \frac{1}{\tilde t_{\delta}} \int_{C_{\tilde t_{\delta}}} dF(x) \\
	&\geq 0,
\end{align*}
by the requirements on $\alpha$ and $\tilde t_{\delta}$. 

In the general case, it is obvious that
\begin{align*}
	\inf\{ J_+( C) : C \subset A, F(C ) \geq \delta\} \leq \inf\{J_{+}(C_{t}): t \text{ such that } F(C_{t}) \geq \delta\}. 
\end{align*}
Moreover, for any $t \geq 0$ such that $F(C_{t}) \leq \delta$ the same arguments as above yield
\begin{align*}
	\int_{A} (\alpha(x) - I\{x \in C_{t}\}) \tilde dF(x) \geq \frac{1}{t} \int_{A} \alpha(x) dF(x) - \frac{1}{t} F(C_{t}) \geq 0,
\end{align*}
for all $\alpha \in \mathcal{F}_{A}$ with $\int_{A} \alpha(x) dF(x) \geq \delta$. We conclude that
\begin{align*}
 \inf \biggl \{ \int _A \alpha (x) d\tilde{F}(x): \alpha \in \mathcal{F} _A, \int _A \alpha(x) dF(x) \geq \delta \biggr \} \geq \tilde F(C_{t}),
\end{align*}
and it follows that
\begin{align*}
	\inf\{ J_+( C) : C \subset A, F(C ) \geq \delta\} \geq J_{+}(C_{t}).
\end{align*}
The proof is completed by taking supremum over $t$ such that $F(C_{t}) \leq \delta$. 
\end{proof}

Denote by $\gamma^+ _{\epsilon}$ the function
\begin{align*}
	\gamma ^+ _{\epsilon}(s) = (1+\epsilon) s \log(1+\epsilon) + (1-(1+\epsilon)s)\log\Big(\frac{1-(1+\epsilon)s}{1-s}\Big).  
\end{align*} 
so that the expression in Lemma \ref{lemma:rateExp} coincides with $\gamma^+ _{\epsilon}(\tilde F( C))$. Then, 
\begin{align*}
	I(A^+ _{\epsilon, \delta}) = \gamma^+ _{\epsilon}(\tilde F(C_{\tilde t_{\delta}})), 
\end{align*}
where $\tilde t_{\delta}$ is such that $\tilde F(C_{\tilde t_{\delta}}) = \delta$. 
Observe that $\gamma^+ _{\epsilon}$ is an increasing function. Therefore, a good choice of sampling distribution is one that makes $\tilde F(C_{\tilde t_{\delta}})$ large. 

Next, consider the set $A^- _{\epsilon, \delta}$ in \eqref{eq:alt2}. The following results are obtained precisely as for $A ^+ _{\epsilon, \delta}$; $ J_-$ is the direct analogue of the set function $J_+$.
\begin{lemma} \label{lemma:rateExpNeg}
	Given $C \subset A$ it holds that
\begin{align*}
	&\inf \biggl \{ \re(G \mid \tilde F):  \frac{dG}{d\tilde F}(x) \leq 1- \epsilon \text{ for } x \in C \biggr \} \\ & \quad = (1-\epsilon) \tilde F( C)\log(1-\epsilon) + (1-(1-\epsilon)\tilde F( C))\log\Big(\frac{1-(1-\epsilon)\tilde F(C )}{1-\tilde F( C)}\Big),
\end{align*}
where the infimum is attained for the probability measure $G_{*}$ with 
\begin{align*}
	\frac{dG_{*}}{d\tilde F}(x) &= 1-\epsilon, \text{ for } x \in C, \\ \frac{dG_{*}}{d\tilde F}(x) &= \frac{1-(1-\epsilon)\tilde F(C )}{1-\tilde F( C)}, \text{ for } x \in C^{c}.
\end{align*}
\end{lemma}
\begin{lemma}
	\begin{align*}
	I(A^- _{\epsilon, \delta}) = \inf\{ J_-( C) : C \subset A, F(C ) \geq \delta\}. 
\end{align*}
\end{lemma}
Let $\gamma ^- _{\epsilon}$ be the function
\begin{align*}
	\gamma^- _{\epsilon}(s) = (1-\epsilon) s \log(1-\epsilon) + (1-(1-\epsilon)s)\log\Big(\frac{1-(1-\epsilon)s}{1-s}\Big).  
\end{align*} 
This is an increasing function in $s$ and thus the optimal $C$ is given precisely by the $C_{\tilde t_{\delta}}$ in Proposition \ref{prop:Set}. Thus, $I(A^- _{\epsilon, \delta}) = \gamma ^- _{\epsilon} (\tilde{F} (C _{\tilde t_{\delta}}))$, whenever there exists $\tilde t_{\delta}$ such that $F(C_{\tilde t_{\delta}}) = \delta$. 

Having obtained an explicit characterization of the rate function we now study some well-known examples in order to illustrate the performance analysis.

\begin{example}\label{ex:MC1}
Consider a standard Monte Carlo algorithm, i.e., $\tilde F = F$, and let $A$ be a set with $F(A) = p$. Put
$$ A ^+ _{\epsilon, \delta} = \Big\{ G \in \M _1 : \frac{dG}{dF}(x) \geq 1+\epsilon \text{ for } x \in \,\text{some } C \subset A, F( C) \geq \delta\Big\}.$$ 
For standard Monte Carlo, the rate function associated with Sanov's theorem is the relative entropy: $I ^{MC}(G) = \re(G\mid F)$. The rate can be computed, just as in the general importance sampling case, as
\begin{align*}
	I^{MC}(A^{+}_{\epsilon, \delta}) = \inf\{J^{MC}_{+}(C ) : C \subset A, F( C) \geq \delta\},
\end{align*}
	with 
\begin{align*}
	J^{MC}_{+}(C ) = \inf \Bigl \{\re(G \mid F): \frac{dG}{dF}(x) \geq 1+\epsilon, \ x \in C \Bigr\} = \gamma^{+}_{\epsilon}(F( C)).
\end{align*}
Suppose there exists a set $C$ such that $F(C) = \delta$. Then, since $\gamma^{+}_{\epsilon}$ is increasing, we conclude that
\begin{align*}
	I^{MC}(A^{+}_{\epsilon, \delta}) = \gamma^{+}_{\epsilon}(\delta), 
\end{align*}
and, by the reasoning leading to \eqref{eq:nCramer}, the number of samples needed in order to obtain a specific error probability is proportional to the reciprocal of the rate $\gamma ^+ _{\epsilon}(\delta)$, with the proportionality constant depending only on the desired accuracy, i.e., bound on the error probability.
 
Let $I^{IS}$ denote the rate function of Theorem \ref{thm:Laplace} and suppose that in Proposition \ref{prop:Set} such a $\tilde{t}_{\delta}$ exists. An importance sampling algorithm with sampling distribution $\tilde F$ has rate
\begin{align*}
	I ^{IS}(A^{+}_{\epsilon, \delta}) = \gamma^{+}_{\epsilon}(\tilde F(C_{\tilde t_{\delta}})).
\end{align*}
If the cost for generating one sample from $\tilde F$ is $c$ times the cost for generating one sample from $F$, then the reduction in computational cost is roughly
\begin{align*}
	c \frac{n^{IS}}{n^{MC}} \approx c \frac{I^{MC}(A^{+}_{\epsilon, \delta})}{I^{IS}(A^{+}_{\epsilon, \delta})} = c \frac{\gamma^{+}_{\epsilon}(\delta)}{\gamma^{+}_{\epsilon}(\tilde F(C_{\tilde t_{\delta}}))}. 
\end{align*}
\end{example}

\begin{example}[Light-tailed random walk] \label{ex:IS1}
For some $m \in \N$, let $X_1,\dots ,X_m$ be independent and identically distributed with finite moment generating function and let $F$ denote the distribution the vector $(X_1,\dots,X_m)$. Consider the normalized  light-tailed random walk $S_{m}/m $ where $S_m = (X_{1}+\dots+X_{m})$. Take $A = \{x: \sum x_i \geq am \}$ to be the set of interest, so that $F(A) = \Prob(S_{m} \geq ma)$, and take the sampling distribution $\tilde F_{\theta}$ as an exponential change of measure: 
\begin{align*}
	\frac{dF}{d\tilde F_{\theta}}(x_{1}, \dots, x_{m}) = e^{m(\kappa(\theta) - \theta \frac{1}{m}\sum_{i=1}^{m} x_{i})}.
\end{align*}
Here $\kappa$ is the logarithmic moment generating function of $X_{1}$. Consider the set $A^{+}_{\epsilon, \delta}$ in \eqref{eq:alt1}.  From Proposition \ref{prop:Set} the rate $I(A^{+}_{\epsilon, \delta})$ is given by 
$\gamma^{+}_{\epsilon}(\tilde F_{\theta}(C_{\tilde t_{\delta}}))$, where
\begin{align*}
	C_{\tilde t_{\delta}} = \Big\{x \in A: \frac{1}{m}\sum_{i=1}^{m} x_{i} \leq (\kappa(\theta)-(1/m)\log \tilde t_{\delta})/\theta\Big\}. 
\end{align*}
By the choice of $\tilde t_{\delta}$, 
\begin{align*}
	\tilde F_{\theta}(C_{\tilde t_{\theta}}) &= \E[ e^{\theta S_{m} - m\kappa(\theta)}I\{ a \leq S_{m}/m \leq (\kappa(\theta)-(1/m)\log \tilde t_{\delta})/\theta\}] \\
& \geq  e^{m(\theta a - \kappa (\theta) } \E[I\{ a \leq S_m /m  \leq (\kappa(\theta)-(1/m)\log \tilde t_{\delta})/\theta\}] \\
&= e^{m (\theta a - \kappa (\theta))} F(C_{\tilde t_{\delta}}) \\
	& = e^{m (\theta a - \kappa (\theta))} \delta.  
\end{align*}
If the cost for each replication of the importance sampling algorithm is $c$ times the cost for each replication of the standard Monte Carlo algorithm, then we conclude that the reduction in computational cost is given by
\begin{align*}
	c \frac{\gamma^{+}_{\epsilon}(\delta)}{\gamma^{+}_{\epsilon}(\tilde F(C_{\tilde t_{\delta}}))} \leq  c\frac{\gamma^{+}_{\epsilon}(\delta)}{\gamma^{+}_{\epsilon}(e^{m (\theta a - \kappa (\theta))} \delta))}.
\end{align*}
A good choice of $\theta$ is the $\theta _a$ that maximises $\theta a - \kappa(\theta)$. It holds that this $\theta _a$ is given by  $\kappa '(\theta_a) = a $. We emphasize that in addition to suggesting the well known exponential change of measure with parameter $\theta_{a}$, the large deviation analysis also provides this useful upper bound on the reduction in computational cost. 
\end{example}

\subsection{Applications in rare-event simulation}
The efficiency analysis of the importance sampling algorithms presented so far is not targeted specifically to capture the performance of rare-event simulation algorithms. In this section we illustrate how a rare-event analysis can be performed, based on Theorem \ref{thm:Laplace} and the suggested performance criteria. The elementary examples presented in this section demonstrate that, based on Theorem \ref{thm:Laplace}, one can obtain results that are consistent with variance analysis as well as estimates on the sample size required to reach a certain precision and the reduction in computational cost from using importance sampling.  

We begin by analyzing standard Monte Carlo and the importance sampling algorithm for a light-tailed random walk with an emphasis on rare events.

\begin{example}[Rare-event analysis for standard Monte Carlo] \label{ex:MC}
Consider computing a probability $F(A) = p$ by standard Monte Carlo, as in Example \ref{ex:MC1}. Take $\delta = \delta' p$ for some $\delta' \in (0,1)$. The performance of the algorithm can be captured by the rate $I^{MC}(A^{+}_{\epsilon, \delta'p}) = \gamma^{+}_{\epsilon}(\delta' p)$, where
\begin{align*}
	\gamma^{+}_{\epsilon}(\delta' p) &=   p \delta'(1 + \epsilon) \log (1+ \epsilon) + \left(1 - p \delta'(1+\epsilon) \right) \log \left( \frac{1 - p\delta'(1 + \epsilon)}{1-p\delta'} \right)\\
	&= p \delta'[(1 + \epsilon) \log (1+ \epsilon) -\epsilon] + o(p^{2}). 
\end{align*}
Thus, as $p \to 0$, the rate decays linearly with $p$ and as a consequence the sample size needed for a given precision increases proportionally to $1/p$. 
\end{example}

\begin{example}[Rare-event analysis for a light-tailed random walk] \label{ex:IS}
In the random walk example, Example \ref{ex:IS1}, the performance of the importance sampling algorithm, with parameter $\theta = \theta_{a}$, is captured by $\gamma^{+}_{\epsilon}(e^{m(\theta_{a}a-\kappa(\theta_{a}))}\delta)$. Let $p_{m} = \Prob(S_{m} \geq ma)$ and take $\delta = \delta' p_{m}$. It follows that
\begin{align*}
	\gamma^{+}_{\epsilon}(e^{m(\theta_{a}a-\kappa(\theta_{a}))}\delta'p_{m}) 
	&= e^{m(\theta_{a}a-\kappa(\theta_{a}))} p_{m} \delta'[(1 + \epsilon) \log (1+ \epsilon) -\epsilon] \\ &\quad + o(e^{2m(\theta_{a}a-\kappa(\theta_{a}))}p_{m}^{2}).
\end{align*} 
The reduction in computational cost for the importance sampler vs.\ the standard Monte Carlo algorithm is then bounded from above by
\begin{align*}
 c\frac{\gamma^{+}_{\epsilon}(\delta' p_{m})}{\gamma^{+}_{\epsilon}(e^{m (\theta_a a - \kappa (\theta_a))} \delta' p_{m}))} \sim c e^{-m(\theta_{a}a-\kappa(\theta_{a}))}, \quad \text{ as } m\to \infty. 
\end{align*}
The conclusion is that the reduction in computational cost is exponential in $m$. 
\end{example}

We end this section by demonstrating the performance of the the zero-variance change of measure \cite[p.\ 127]{asmussen2007}. This choice of sampling distribution is optimal for estimating a probability in the sense that the variance of the estimator is zero; it is often used as a reference point for designing efficient sampling algorithms.
\begin{example}[Zero-variance change of measure]
\label{ex:zeroVar}
Consider the probability $p = F(A)$, for some $A \subset \calX$ and distribution $F$ and  take the importance function to be $I\{x \in A\}$. The zero-variance sampling distribution is the distribution $\tilde {F}$ given by
$$ \frac{d\tilde F}{dF} (x) = \frac{I  \{ x \in A\} }{p}.$$

Let $\delta' \in (0,1)$ and $\delta = \delta'p$. The likelihood ratio is constant over the entire set $A$ and $\tilde{F}(C) = F(C) p ^{-1}$ for any $C \subset A$. That is, $\tilde{F}(C) = F(C\mid A)$, the conditional probability under $F$ of $C$ given $A$. It follows that 
\begin{align*}
	J_{+}( C) = \inf\{\re(G \mid \tilde F): \frac{dG}{d\tilde F}(x) \geq 1+\epsilon, x \in C\} = \gamma^{+}_{\epsilon}(\tilde F( C)) = \gamma^{+}_{\epsilon}\Big(\frac{F( C)}{p}\Big), 
\end{align*}
and the rate
\begin{align*}
	I(A^{+}_{\epsilon, \delta'p}) = \gamma^{+}_{\epsilon}(\delta').
\end{align*}
In particular, the rate $I(A^+ _{\epsilon, \delta p'})$ is independent of $p$. 

For $\delta' = 1$, the case of  having a relative error $\epsilon$ on the estimate of $p$, the rate corresponding to the zero-variance change of measure is $ \infty$. This follows from the fact that no probability measure $G$ which is absolutely continuous with respect to $\tilde{F}$ on $A$, and thus must satisfy $G(A) = 1$, can give rise to such an error. In this particular case the rate ($\infty$) can be obtained without using large deviation results. Since good performance corresponds to a large rate, the zero-variance change of measure is clearly optimal also from the large deviation perspective brought forward in this paper.
\end{example}
\subsection{Performance analysis for computing quantiles}
Let the underlying space $\calX$ be the real line and consider computing a quantile of a distribution $F$ on $\R$. For $\alpha \in (0,1)$ the $\alpha$-quantile of a finite measure $\nu$ on $\R$ is defined by
\begin{align*}
	\Phi_{\alpha}(\nu) = \inf\{x: \nu\bigl((x,\infty)\bigr) \leq \alpha\}. 
\end{align*} 
For $\epsilon > 0$, consider the set
\begin{align*}
	A^{+}_{\epsilon} = \{\nu \in \M: \Phi_{\alpha}(\nu) \geq (1+\epsilon) \Phi_{\alpha}(F)\}. 
\end{align*}
Let $\tilde F$ be the sampling distribution of an importance sampling algorithm for computing $\Phi_{\alpha}(F)$. Suppose that the importance function is $I\{x \in (a,\infty)\}$ where $a \leq \Phi_{\alpha}(F)$ and let $I$ be the rate function in Theorem \ref{thm:Laplace}. The performance of the importance sampling algorithm can be quantified by $I(A^{+}_{\epsilon})$. Of course one may also consider $A^{-}_{\epsilon}$ defined in the obvious way, but for this illustration we work exclusively with $A^{+}_{\epsilon}$. 

Let us compute the rate $I(A^{+}_{\epsilon})$. First note that, with $q_{\alpha, \epsilon} = (1+\epsilon)\Phi_{\alpha}(F)$, 
\begin{align*}
	A^{+}_{\epsilon} &= \{\nu \in \M: \Phi_{\alpha}(\nu) \geq q_{\alpha, \epsilon}\} 
	= \{\nu \in \M:  \nu(q_{\alpha, \epsilon}, \infty) \geq  \alpha\}
\end{align*}
and the rate is therefore given by
\begin{align*}
	I(A^{+}_{\epsilon}) &= \inf\{\re(G \mid \tilde F): \nu = \Psi(G), \ \nu(q_{\alpha, \epsilon}, \infty) \geq \alpha\} \\
	 &= \inf \left \{\re(G \mid \tilde F):  \int _{\R} I\{x > q_{\alpha, \epsilon}\} w(x) dG(x) \geq \alpha \right \}.  
\end{align*}
The infimum is attained at $G^{*}_{\alpha}$ given by
\begin{align*}
	\frac{dG^{*}_{\alpha}}{d\tilde F}(x) = \frac{e^{\lambda k(x)}}{M(\lambda)}, 
\end{align*}
where $k(x) = I\{x > q_{\alpha, \epsilon}\} w(x)$, $M(\lambda) = \int e^{\lambda k(x)} \tilde F(dx)$  and $\lambda$ is given as the solution to 
\begin{align}\label{eq:lambda}
	\alpha = \frac{\partial_{\lambda} M(\lambda)}{M(\lambda)}. 
\end{align}
To see that the infimum is attained at $G^{*}_{\alpha}$, note that by the variational formula for relative entropy \cite[Proposition 4.5.1]{Dupuis97}, for all $\lambda \geq 0$ and $G \in \M_{1}$ such that $\int k(x) dG(x) \geq \alpha$,
\begin{align*}
	\re(G \mid \tilde F) \geq \lambda \alpha - \log M(\lambda),
\end{align*}
and the inequality is satisfied with equality for $G^{*}_{\alpha}$. We have just proved the following. 
\begin{proposition}\label{prop:quantile}
$I(A^{+}_{\epsilon}) = \lambda \alpha - \log M(\lambda)$ where $\lambda$ is determined by \eqref{eq:lambda}. 
\end{proposition}

\begin{example} 
For a standard Monte Carlo algorithm the rate $I(A^{+}_{\epsilon})$ can be explicitly computed. Indeed, in this case $k(x) = I\{x \geq q_{\alpha, \epsilon}\}$ and
\begin{align*}
	M(\lambda) = e^{\lambda} p_{\alpha, \epsilon} + 1- p_{\alpha, \epsilon},
\end{align*}
with $p_{\alpha, \epsilon} = F\bigl((q_{\alpha, \epsilon}, \infty)\bigr) \leq \alpha$. The equation for $\lambda$ becomes
\begin{align*}
	\alpha = \frac{\partial_{\lambda} M(\lambda)}{M(\lambda)} = \frac{e^{\lambda}p_{\alpha, \epsilon}}{e^{\lambda}p_{\alpha, \epsilon} + 1-p_{\alpha, \epsilon}}, 
\end{align*}
which leads to
\begin{align*}
	\lambda = \log\Big(\frac{\alpha(1-p_{\alpha, \epsilon})}{(1-\alpha)p_{\alpha, \epsilon}}\Big)
\end{align*}
and finally
\begin{align*}
	I(A^{+}_{\epsilon}) &= \lambda \alpha - \log M(\lambda) = \alpha \log\Big(\frac{\alpha}{p_{\alpha, \epsilon}}\Big) + (1-\alpha) \log\Big(\frac{1-\alpha}{1-p_{\alpha, \epsilon}}\Big).
\end{align*}
That is, if $\re(\alpha \mid p_{\alpha, \epsilon})$ refers to the relative entropy between two Bernoulli distributions with parameters $\alpha$ and $p_{\alpha, \epsilon}$, respectively, then
\begin{align*}
	I(A^{+}_{\epsilon}) = \re(\alpha \mid p_{\alpha, \epsilon}).
\end{align*}
\end{example}
For a general importance sampling algorithm the expression for $I(A^+ _{\epsilon})$ in Proposition \ref{prop:quantile} has to be worked out on a case-by-case basis.

\section[Proof of the Laplace principle]{Proof of  Theorem \ref{thm:Laplace}} \label{sec:proofs}

In this section the proof of the Laplace principle for the weighted empirical measures of importance sampling is presented. The proof relies on the weak convergence approach developed by Dupuis and Ellis \cite{Dupuis97}. The three main steps of the proof are:\\
\begin{enumerate}
\item Derive a representation formula for the pre-asymptotic expectation
$$ W^n = - \frac{1}{n}\log \widetilde\E [e^{-nh(\edfis ^{wf} _n)}].$$
This is achieved by formulating a stochastic control problem that has a minimal cost function equal to $W^n$. In the setting considered here the representation formula reads
\begin{equation}
\label{eq:Rep}
	W ^n = \inf _{\{ G_{n,j} \}} \Eb \biggl[\frac{1}{n} \sum _{j=0} ^{n-1} \re (G_{n,j}(\cdot \mid \bar{\empd} _{n,j}) \mid \tilde{F}) + h(\bar{\empd} _n) \biggr],
\end{equation} 
where $\bar{\empd} _{n,j}$ is the controlled process (empirical measure), $\bar{\empd} _{n,j+1} = \bar{\empd} _{n,j} +  \frac{1}{n} \delta _{\bar{X} _{n,j} } $ and $\bar{\empd} _n = \frac{1}{n} \sum _{j=0} ^{n-1} \delta _{\bar{X} _{n,j}} $, obtained by sampling $\bar{X} _{n,j}$ from the distribution (control) $G_{n,j} (\cdot \mid \bar{\empd} _{n,j})$.
\item The representation formula is used to prove the Laplace principle lower bound,
$$ \liminf _n \frac{1}{n} \log \widetilde \E[e^{-nh(\edfis _n ^{wf})}] \geq - \inf _{G \in \Delta \cap \Gamma} \{ h(\Psi (G)) + \re(G \mid \tilde{F}) \}.$$
\item The most involved step is to use the representation formula to prove the Laplace principle upper bound,
$$ \limsup _n \frac{1}{n} \log \widetilde \E[e^{-nh(\edfis _n ^{wf})}] \leq - \inf _{G \in \Delta \cap \Gamma} \{ h(\Psi (G)) + \re(G \mid \tilde{F}) \}.$$
\end{enumerate}
%%%%%%%%%%%%%%%%%%%%
%Much of the proof of Sanov's theorem \cite{Dupuis97} goes through also for the case of weighted empirical measures and we point out the differences in the arguments. The main new difficulties in proving Theorem \ref{thm:Laplace} are that the mapping $\Psi$ is defined on a subset of $\M_{1}$ and may be unbounded. The first point is handled by making minor adaptions to the arguments in \cite{Dupuis97}, whereas the unboundedness of $\Psi$ is mainly treated in Lemmas \ref{lemma:Step5}-\ref{lemma:Mapping}. In order to make the paper self-contained results and constructions similar to those of \cite{Dupuis97} are included and the corresponding references provided. To make comparisons easier the notation is, for the most part, consistent with \cite{Dupuis97}. 

\subsection{Representation formula}
In this section we show the representation formula \eqref{eq:Rep} for the pre-limit expectation
$$ W^n = - \frac{1}{n}\log \widetilde \E [e^{-nh(\edfis ^{wf} _n)}]. $$
This is achieved by considering a related stochastic control problem. For a more thorough discussion see \cite[Section 2.3]{Dupuis97}.

The difference between the current setting and that of Sanov's theorem is that only a subset $\Gamma = \{ G \in \M_{1} : G(wf) < \infty \}$ of the space of probability measures is under consideration. The proof of the representation formula $W^n$ is therefore very similar to the standard case. Recall that $\edfis _n ^{wf}(g) = \edfis _n (wfg)$ for each measurable $g$. In particular,
$$ \widetilde \E [\edfis ^{wf} _n (I _{\calX})] = \widetilde \E [\edfis _n (wf)] = F(f) < \infty,$$
and it follows that $\edfis _n \in \Gamma$ with probability 1. Let $\Gamma_{n} = \Gamma$ and define, for $j=0, \dots, n-1$,  recursively the sets $\Gamma_{j} \subset \M_{j/n}$ by 
\begin{align*}
	\Gamma_{j} = \{G \in \M_{j/n}: \tilde F(\{y: G + \frac{1}{n}\delta_{y} \in \Gamma_{j+1}\}) = 1\}. 
\end{align*}
That is, if $G \in \Gamma_{j}$, then sampling $Y$ from $\tilde F$ implies that $G+n^{-1}\delta_{Y} \in \Gamma_{j+1}$ with probability $1$. 

%Furthermore, if $G$ and $\tilde{G}$ are both probability measures in $\Gamma$, and $X$ has distribution $\tilde{G}$, then $\frac{1}{2}G + \frac{1}{2}\delta _X$ is in $\Gamma$ with probability 1.

Let $\M_{0}$ be the set containing only the null measure and introduce the measurable mapping $W^n: \bigcup_{k=0}^{n} ( \{k\} \times \M _{k/n}) \to \bar{\R}= [-\infty, \infty]$ by
\begin{align}
\label{eq:Wnj}
	W ^n (j,G) = \begin{cases} - \frac{1}{n}\log \widetilde \E \left[e^{-nh(\Psi(\edfis _n))}\mid \edfis _{n,j} = G \right], &\ \textrm{for } G\in \Gamma_{j}, \\
	\infty, &\ \textrm{for } G \in \Gamma_{j}^c, \end{cases}
\end{align}
for $j=0,...,n-1$ and
\begin{equation}
\label{eq:Wnn}
	W ^n (n,G) = \bar{h}(G), 
\end{equation}
where
\begin{align}
\bar{h}(G) = 
\begin{cases} 
	h(\Psi(G)), &\ \textrm{for } G \in \Gamma, \\
	\infty, &\ \textrm{for } G \in \Gamma ^c.
\end{cases}
\end{align}
In particular, we set $W^n = W^n (0,0)$. 
Note that in (\ref{eq:Wnj}) $G \in \M _{j/n} (\calX)$ is a subprobability measure, and the $\edfis _{n,j} = (1/n) \sum _{i=0} ^{j-1} \delta _{\tilde{X}_i}$ are the empirical subprobability measures obtained by sampling the $\tilde{X}_i$'s from $\tilde{F}$. By the same argument as in \cite{Dupuis97}, the Markov chain $\edfis _{n,j}$ can be used to obtain the recursion formula
\begin{align*}
	W^n (j,G) = - \frac{1}{n}\log \int  e^{-nW^n(j+1, G + \frac{1}{n}\delta _x)} d \tilde{F}(x).
\end{align*}
Since $h$ is bounded and continuous the mapping $W ^n$ is measurable, bounded from below, and bounded from above on $\Gamma$. Together with the recursion formula above, Proposition 4.5.1 in \cite{Dupuis97} gives that $W^n(j,G)$ can be written as
\begin{equation}
\label{eq:Wn_1}
 W^n (j,G)= \inf _{\tilde{G} \in \Delta} \left \{ \frac{1}{n} \re (\tilde{G} \mid \tilde{F}) + \int W^n(j+1, G + \frac{1}{n}\delta _x) d \tilde{G}(x) \right\},
\end{equation}
for $j=0,...,n-1$, where
$$ \Delta = \{ G \in \M _1 : \re (G \mid \tilde{F}) < \infty \}.$$
Moreover, the infimum is attained at $\tilde{G}_{n,j} \in \Delta$ defined by the Radon-Nikodym derivative
\begin{equation}
\label{eq:min_control}
	\frac{d\tilde{G}_{n,j}}{d\tilde{F}}(x) = \frac{e^{-W^n (j+1, G + \frac{1}{n}\delta _x)}}{\int e^{-W^n (j+1, G + \frac{1}{n}\delta _y)} d\tilde{F}(y)}.
\end{equation}
To derive the representation formula for $W^n$, consider the following related stochastic control problem. For $n \in \N$ and $j \in \{ 0,1,...,n\}$, let $G_{n,j}$ be a stochastic kernel on $\calX$ given $\M _{j/n}$. A controlled process $\{ \bar{\empd} _{n,j} \}$ is defined by $\bar{\empd} _{n,0} = \edfis _{n,0} = 0$ and
\begin{equation*}
	\bar{\empd} _{n,j} = \frac{1}{n} \sum _{k=0} ^{j-1} \delta _{\bar{X}_{n,k}}, \ \ \ \bar{\empd} _n = \bar{\empd} _{n,n}, 
\end{equation*}
where the conditional distribution of $\bar{X} _{n,k}$ given $\bar{\empd} _{n,0}, \bar{\empd} _{n,1},...,\bar{\empd} _{n,k}$ is
\begin{equation*}
	\bar{\Prob} (\bar{X} _{n,k} \in dx \mid \bar{\empd} _{n,0}, \bar{\empd} _{n,1},...,\bar{\empd} _{n,k}) = G_{n,k} (dx \mid \bar{\empd} _{n,k}).
\end{equation*}
All random variables and the corresponding (controlled) empirical subprobability measures are for all $n$ defined on a common probability space $(\bar{\Omega},\bar{\mathcal{F}}, \bar{\Prob})$ which will be used throughout the paper. For these dynamics, $j \in \{ 0,1,...,n\}$ and $G \in \M _{j/n}$, define the minimal cost functions 
\begin{equation}
\label{eq:rep_formula}
	\bar{W}^n(j,G) = \inf _{ \{ G_{n,j} \}} \Eb \biggl[ \frac{1}{n} \sum _{k=j} ^{n-1} \re(G_{n,k} (\cdot \mid \bar{\empd} _{n,k})\mid \tilde{F}) + \bar{h}(\bar{\empd} _n) \mid \bar{\empd} _{n,j} = G \biggr].
\end{equation}
For $j=0$ and $G = 0$ we set
\begin{align*}
	\bar{W}^n = \bar{W} ^n (0,0) = \inf _{ \{ G_{n,j} \}} \Eb \biggl[ \frac{1}{n} \sum _{k=0} ^{n-1} \re(G_{n,k} (\cdot \mid \bar{\empd} _{n,k})\mid \tilde{F}) + \bar{h}(\bar{\empd} _n) \biggr].
\end{align*}
\begin{proposition}
\label{prop:rep_formula}
	Let $\bar{W} ^n$ be given by (\ref{eq:rep_formula}) and $W^n$ be the solution to (\ref{eq:Wnj}) and (\ref{eq:Wnn}). Then $W^n = \bar{W} ^n$.
\end{proposition}
\begin{proof} We begin by proving that $\bar{W} ^n (j,G) \geq W ^n (j,G)$ using backwards induction on $j$. Fix a control sequence $\{ G_{n,j} \}$ and let $\{ \bar{\empd}_{n,j} \}$ denote the associated controlled process. Let $\bar \Gamma _n = \Gamma$ and define sets $\bar \Gamma _j$ associated with the control sequence $\{ G_{n,j} \}$ by
$$ \bar \Gamma _j = \{ G \in \M _{j/n}: G_{n,j}(\{y: G + \frac{1}{n}\delta _y \in \bar \Gamma _{j+1} \} | G)=1 \}, \ j=0,\dots, n-1.$$
The definition is such that if at time $j$ the controlled process $\bar{\empd}_{n,j}$ lies in $\bar \Gamma _j$, then by sampling from $G_{n,j}, G_{n,j+1},...,G_{n, n-1}$ the controlled process $\bar{\empd}_n$ will belong to $\Gamma$ with probability 1.

Consider first the case $j=n$. Clearly, 
$$ \bar{W}^n (n,G) = \bar{h}(G) = W ^n (n,G),$$
and the claim is trivial for this $j$.

Take $j=n-1$. Suppose that $G \in \bar \Gamma _{n-1}$, so that $G_{n,n-1}(\{ y: G + \frac{1}{n}\delta _y \in \bar \Gamma _n)\}\mid G) = 1$. Using \eqref{eq:Wnj}, \eqref{eq:Wnn} and \eqref{eq:Wn_1}, we have
\begin{align*}
	& \Eb \biggl[\frac{1}{n}\re(G_{n,n-1}(\cdot \mid \bar{\empd}_{n,n-1})\mid \tilde{F}) + \bar{h}(\bar{\empd}_n) \mid \bar{\empd}_{n,n-1} = G\biggr] \\
	&\quad = \Eb \biggl[\frac{1}{n}\re(G_{n,n-1}(\cdot \mid \bar{\empd}_{n,n-1})\mid \tilde{F}) + \bar{h}(\bar{\empd}_n) + \int W^n (n, \bar{\empd} _n) dG_{n,n-1} \\
	&\qquad \qquad - \int W^n (n, \bar{\empd} _n) dG_{n,n-1} \mid \bar{\empd}_{n,n-1} = G \biggr] \\
	& \quad \geq \Eb \biggl[W^n (n-1,\bar{\empd} _{n,n-1}) + \bar{h}(\bar{\empd} _n) - \int W^n (n, \bar{\empd} _n) dG_{n,n-1} \mid \bar{\empd}_{n,n-1} = G \biggr] \\
	&\quad = W ^n (n-1, G).
\end{align*}
If instead $G \in \bar \Gamma _{n-1} ^c$, then $G_{n,n-1}(\{ y: G + \frac{1}{n}\delta _y \in \Gamma _n)\}\mid G) < 1$ and since $\bar{h}$ is infinite on $\bar \Gamma _n ^c$, this implies
$$ \Eb \biggl[\frac{1}{n}\re(G_{n,n-1}(\cdot \mid \bar{\empd} _{n,n-1}) \mid \tilde{F}) + \bar{h}(\bar{\empd}_n) \mid \bar{\empd}_{n,n-1} = G \biggr] = \infty .$$
This shows that
$$ \Eb \biggl[\frac{1}{n}\re(G_{n,n-1}(\cdot \mid \bar{\empd} _{n,n-1}) \mid \tilde{F}) + \bar{h}(\bar{\empd}_n) \mid \bar{\empd} _{n,n-1} = G\biggr] \geq W^n (n-1, G),$$
for any choice of $G$. 

Proceeding similarly for $j=n-2, n-3,...,0$ shows that
$$ \Eb \biggl[\frac{1}{n} \sum _{k=j} ^{n-1} \re(G_{n,k}(\cdot \mid \bar{\empd} _{n,k})\mid \tilde{F}) + \bar{h}(\bar{\empd}_n) \mid \bar{\empd}_{n,j} = G\biggr] \geq W^n (j, G),$$
for all $G$ and $j$. Taking infimum over all admissible control sequences $\{ G_{n,j} \} \subset \Delta$ proves the inequality.

Next, the reverse inequality $\bar{W}^n (j,G) \leq W^n (j,G)$ is proved. Consider the control sequence $\{ \tilde{G} _{n,j}\} $ defined by (\ref{eq:min_control}). For this particular sequence it holds that $\bar \Gamma _j = \Gamma_{j}$ for all $j$.

The case $j=n$ was handled above and thus we start by considering $j=n-1$. If $G \in \Gamma_{n-1}$, then by the definition of $\tilde{G}_{n,n-1}$ and \eqref{eq:Wn_1},
\begin{align*}
	& \Eb \biggl[\frac{1}{n}\re (\tilde{G} _{n,n-1}(\cdot \mid \bar{\empd} _{n,n-1}) \mid \tilde{F}) + \bar{h}(\bar{\empd} _n) \mid \bar{\empd} _{n,n-1} = G \biggr]\\
	&\quad = \Eb \biggl[\frac{1}{n}\re (\tilde{G} _{n,n-1} (\cdot \mid \bar{\empd} _{n,n-1}) \mid \tilde{F}) + W^n (n, \bar{\empd} _n) \mid \bar{\empd} _{n,n-1} = G \biggr] \\
	&\quad = \Eb \biggl[\frac{1}{n}\re (\tilde{G} _{n,n-1} (\cdot \mid \bar{\empd} _{n,n-1}) \mid \tilde{F}) \mid \bar{\empd} _{n, n-1} = G \biggr] \\ 
	&\qquad + \Eb \bigl[ W^n (n, \bar{\empd} _{n}) \mid \bar{\empd} _{n, n-1} = G \bigr] \\
	& \quad =\frac{1}{n}\re (\tilde{G} _{n,n-1} (\cdot \mid G) \mid \tilde{F}) + \int W^n (n, G + \frac{1}{n}\delta _y) d \tilde{G}_{n,n-1}(y \mid G) \\
	&\quad= W^n(n-1, G).
\end{align*}
If instead $G \in \Gamma_{n-1} ^c$, then $\tilde{G}_{n,n-1}(\{ y: G + \frac{1}{n}\delta _y \in \Gamma  \}\mid G) < 1$ and since $\bar{h}$ is infinite on $\Gamma ^c$ we have
$$ \Eb \biggl[\frac{1}{n}\re (\tilde{G} _{n,n-1} (\cdot \mid \bar{\empd} _{n,n-1}) \mid \tilde{F}) + \bar{h}(\bar{\empd} _n) \mid \bar{\empd} _{n,n-1} = G \biggr] = \infty = W^n (n-1, G).$$
This shows that 
$$ \Eb \biggl[\frac{1}{n}\re (\tilde{G} _{n,n-1} (\cdot \mid \bar{\empd} _{n,n-1}) \mid \tilde{F}) + \bar{h}(\bar{\empd} _n) \mid \bar{\empd} _{n,n-1} = G \biggr] = W^n(n-1, G),$$
for all $G$. Proceeding similarly for $j=n-2, ... ,0$ shows that 
$$ \Eb \biggl[\frac{1}{n} \sum _{k=j} ^{n-1} \re (\tilde{G} _{n,k}(\cdot \mid \bar{\empd} _{n,k}) \mid \tilde{F}) + \bar{h}(\bar{\empd} _n) \mid \bar{\empd} _{n,j} = G \biggr] = W^n(j, G),$$
for all $j$ and $G$. Taking infimum over all admissible control sequences $\left \{G_{n,j} \right \}$ in $\Delta$ yields the desired inequality. This completes the proof. 
\end{proof}
%%%%%%%%%%%%%%%%%%%%%%%%%%%%%%%%%%%%%%%%%%%%%%%%%%%%
\subsection{Laplace principle lower bound}
In this section the Laplace principle lower bound,
\begin{equation}
\label{eq:laplaceLower}
 \liminf _n \frac{1}{n}\log \E[e^{-nh(\edfis ^{wf} _n)}] \geq - \inf _{G \in \Delta \cap \Gamma} \{h(\Psi(G)) + \re(G \mid \tilde{F})\},
\end{equation}
is proved.
With $W^n$ as in \eqref{eq:Wnj}, the lower bound is is equivalent to the upper bound
\begin{equation}
\label{eq:Wn_upper}
\limsup _n W^n \leq \inf _{G \in \Delta \cap \Gamma} \{h(\Psi(G)) + \re(G \mid \tilde{F})\},
\end{equation}
which is proved by using the representation formula \ref{eq:rep_formula}) for $W^n$, derived in Proposition \ref{prop:rep_formula} is used.

The following strong law of large numbers will play a role in proving \eqref{eq:Wn_upper}. Note that in Proposition \ref{prop:SLLN_IS} the function $f$ is any measurable, non-negative $F$-integrable function, not necessarily any specific importance function. However, any importance function is of course an example of the type of function used in the proposition. Henceforth, $f$ will again be used denote an (unspecified) importance function.
\begin{proposition}[Strong law of large numbers under importance sampling]
\label{prop:SLLN_IS}
	Let $f$ be non-negative, measurable and $F$-integrable. Let $\{ X_j \}$ be independent and identically distributed with common distribution $F$.  Let $\empd ^f _n$ be the weighted measure in $\M$ determined by
$$ \empd ^f _n (g) = \frac{1}{n} \sum _{j=0} ^{n-1} f(X_j)g(X_j),$$
for each bounded measurable function $g$. Then, with probability 1, 
$$ \empd ^f _n \tauconv F^f,$$
in $\M$. 
\end{proposition}
A proof of the corresponding result for empirical measures, i.e., no weights, is found in \cite{Dupuis97} and the case of weighted measures can be treated in a similar way; we omit the proof.

Fix a probability measure $G \in \Delta \cap \Gamma$. Define the sequence of controls $\{ G _{n,j} \}$ by $G _{n,j} = G$ for each $j=0,1,...,n-1$, so that in every step, the control does not depend on the controlled process. All the $\bar{X} _{n,j}$'s are independent and identically distributed with common distribution $G$ and the associated controlled process $\bar{\empd} _n$ belongs to $\Gamma$ with probability 1. Using the representation formula for $W^n$, it follows that
$$ W^n \leq  \Eb \biggl[\frac{1}{n} \sum _{k=0} ^{n-1} \re (G \mid \tilde{F}) + \bar{h}(\bar{\empd} _n)] = \re (G \mid \tilde{F}) + \Eb [h(\Psi(\bar{\empd} _n)) \biggr].$$
The product $wf$ is non-negative, measurable and $G$-integrable. By the choice $G$ and the construction of the controlled process $\bar{\empd} _n$, Proposition \ref{prop:SLLN_IS} implies that with probability 1, $\Psi (\bar{\empd} _n;\cdot) \tauconv \Psi (G;\cdot)$. By assumption $h$ is bounded and continuous with respect to the $\tau$-topology. Thus, $h(\Psi (\bar{\empd} _n)) \to h(\Psi (G))$ with probability 1. By the dominated convergence theorem, 
$$ \Eb [h(\Psi (\bar{\empd} _n))] \to h(\Psi (G)),$$
and we conclude that for any $G \in \Delta \cap \Gamma$
$$ \limsup _n W^n \leq \re (G \mid \tilde{F}) + h(\Psi (G)).$$
Finally, taking infimum over $G$ in $\Delta \cap \Gamma$ on the right-hand side proves the upper bound \eqref{eq:Wn_upper}, and thus the Laplace principle lower bound.
%%%%%%%%%%%%%%%%%%%%%%%%%%%%%%%%%%%%%%%%%%%%%%%%%%
\subsection{Laplace principle upper bound}
Just as for the lower bound, the Laplace principle upper bound
\begin{equation}
\label{eq:laplaceUpper}
 \limsup _n \frac{1}{n}\log \E[e^{-nh(\edfis ^{wf} _n)}] \leq - \inf_ {G \in \Delta \cap \Gamma} [h(\Psi(G)) + \re(G \mid \tilde{F})],
\end{equation}
can be stated as a lower limit for the minimal cost $W^n$,
\begin{equation}
\label{Wn_lower}
 \liminf _n W^n \geq \inf_{G \in \Delta \cap \Gamma} [h(\Psi(G)) + \re(G \mid \tilde{F})].
 \end{equation}
To prove \eqref{Wn_lower} it is enough to show that every subsequence has a further subsequence that satisfies the lower limit. Therefore, we henceforth work with a fixed subsequence also denoted by $W^n$.

Since $\re (\cdot \mid \tilde{F})$ is a convex function \cite[Propisition 1.4.3]{Dupuis97},
$$ \frac{1}{n} \sum _{j=0} ^{n-1} \re (G_{n,j}(\cdot \mid \bar{\empd} _{n,j}) \mid \tilde{F}) \geq  \re \biggl( \frac{1}{n} \sum _{j=0} ^{n-1} G_{n,j}(\cdot \mid \bar{\empd} _{n,j}) \mid \tilde{F} \biggr), $$ 
and together with the representation formula this gives the lower bound
$$ W ^n \geq \inf _{\{ G_{n,j} \} } \Eb [\re (\bar{G}_n \mid \tilde{F}) + \bar{h} (\bar{\empd} _n)],$$
where $\bar G _n$ is defined as $\bar{G} _n = (1/n) \sum _{j=0} ^{n-1} G_{n,j}(\cdot \mid \bar{\empd} _{n,j})$, 
By optimality, for every $\epsilon > 0$, there exists a control sequence $\{ G_{n,j} \}$ and associated $\bar{G} _n$ such that
\begin{equation}
\label{eq:cont_seq} 
W^n + \epsilon \geq \Eb [\re (\bar{G}_n \mid \tilde{F}) + \bar{h} (\bar{\empd} _n)].
\end{equation} 
There is no restriction on assuming $G_{n,j} \in \Delta \cap \Gamma$ for each $j$ and for the remainder of this section this assumption is made.

To show the Laplace principle upper bound it is now enough to prove the following result.
\begin{proposition}
\label{prop:Upper}
	Every subsequence of $\{ (\bar{G}_n, \bar{\empd}_n ) \}$ has a further subsequence, also denoted $\{ (\bar{G}_n, \bar{\empd}_n ) \}$, such that  $\left (\Psi (\bar{G}_n), \Psi (\bar{\empd}_n) \right) \tauconv  (\Psi (\bar{\empd}), \Psi (\bar{\empd}) )$ with probability 1 along this subsequence, and where $\bar{\empd}$ belongs to $\Delta \cap \Gamma$ with probability 1.
\end{proposition}
Suppose that Proposition \ref{prop:Upper} holds. A proof of the Laplace principle upper bound follows from Fatou's lemma and the lower semi-continuity of the relative entropy mapping $G \mapsto \re (G \mid \tilde{F})$:
\begin{align*}
	\epsilon + \liminf _n W^n &\geq \liminf _n \bar{\E}[\re (\bar{G}_n \mid \tilde{F}) + \bar{h}(\bar{\empd}_n)] \\
	&\geq \bar{\E}[\liminf _n \re (\bar{G}_n \mid \tilde{F}) + \liminf _n h(\Psi(\bar{\empd}_n))] \\
	&\geq \Eb [\re (\bar{\empd} \mid \tilde{F}) + \bar{h} (\bar{\empd})] \\
	&\geq \inf _{G \in \Delta \cap \Gamma} \left \{ \re (G \mid \tilde{F}) + h(\Psi (G))\right \},
\end{align*} 
where in the last step we used that $\bar{\empd}$ is in $\Delta \cap \Gamma$ with probability 1.

Proposition \ref{prop:Upper} is proved by a series of lemmas. The idea is to first work with $\M _1$ equipped with the weak topology and show that $\{ (\bar{G}_n, \bar{\empd}_n ) \}$ is tight in this topology. By showing that $\Eb [\re (\bar{G} _n | \tilde{F})]$ is uniformly bounded (Lemma \ref{lemma:Step1}) the tightness of $\{ (\bar{G}_n, \bar{\empd}_n ) \}$ is obtained by showing tightness of each of the marginals. Prohorov's theorem implies relative compactness and thus each subsequence has a subsubsequence converging to some random element $(\bar{G}, \bar{\empd})$. Lemma \ref{lemma:Step3}, which corresponds to  \cite[Lemma 2.5.1]{Dupuis97}, concludes that $\bar{G} = \bar{\empd}$ w.p.\ 1, thus establishing that w.p.\ 1 for each subsequence, $$ (\bar{G} _n, \bar{\empd} _n) \weak (\bar{\empd}, \bar{\empd}),$$
along some further subsequence. Next, Lemma \ref{lemma:Step4} establishes that the convergences of the marginals $\bar{G}_n$ and $\bar{\empd} _n$ to $\bar{\empd}$ are still valid when $\M _1$ is equipped with the $\tau$-topology. The main ingredient of the proof is an approximation argument introduced in \cite[Lemma 9.3.3]{Dupuis97} and Lemma \ref{lemma:Step4} is a version of that result in the simpler setting where the underlying random variables are independent and identically distributed. 

Once the convergence in $\M _1$ equipped with the $\tau$-topology is established w.p.\ 1, it remains to show that it is preserved under the mapping $\Psi$. Lemma \ref{lemma:Step5} proves that $\bar{\empd}$ is in $\Gamma$ w.p.\ 1, which implies that $\Psi(\bar{\empd};\cdot)$ is well-defined. The main additional difficulty is handled in Lemma \ref{lemma:Mapping} where a truncation argument is used to prove that $\Psi(\bar{G} _n ; \cdot)$ converges to $\Psi(\bar{\empd} ; \cdot)$  and 
$\Psi(\bar{\empd} _n; \cdot )$ converges to $\Psi(\bar{\empd} ; \cdot)$ 
in the $\tau$-topology on $\M$.

\begin{lemma}
\label{lemma:Step1}
	For a sequence $\{ \bar{G} _n \}$ of control sequences such that \eqref{eq:cont_seq} holds and $G_{n,j} \in \Gamma \cap \Delta$ for each $j$ and $n$, it holds that 
$$ \sup _n \bar{\E}[\re (\bar{G}_n \mid \tilde{F})] < \infty. $$
\end{lemma}
\begin{proof} Since $\bar{h}$ is bounded on $\Gamma$, it is possible to find a constant $M < \infty$ such that $\bar{h} $ over $\Gamma$ is bounded my $M$; $\sup _{G \in \Gamma} \mid \bar{h}(G) \mid \leq M$. The choice of control sequence $\{ G_{n,j} \}$ implies that $\bar{\empd} _n \in \Gamma$ with probability 1 for all $n$. Therefore, 
\begin{align*}
\sup _n \Eb [\re (\bar{G} _n \mid \tilde{F})] &= \sup _n \Eb [\re (\bar{G} _n \mid \tilde{F}) - M] + M \\
	&\leq \sup _n \Eb[\re (\bar{G} _n \mid \tilde{F}) + \bar{h}(\bar{\empd} _n)] + M \\
	&\leq \sup _n (W^n + \epsilon) +M \\
	&= \sup _n W^n + \epsilon + M.
\end{align*} 
For each $n$, $\Eb [\re (\bar{G} _n \mid \tilde{F})] < \infty$ and by the proof of the Laplace principle lower bound $\limsup _n W ^n < \infty$, which shows that $\sup _n \Eb [\re (\bar{G} _n \mid \tilde{F})] < \infty$. 
\end{proof}
The uniform boundedness established in Lemma \ref{lemma:Step1} is used to prove the tightness of the sequence of admissible control measures. 
\begin{lemma}
\label{lemma:Step2}
	Under condition (i) of Theorem \ref{thm:Laplace}, the sequence
$$ \left\{ \frac{1}{n}\sum _{i=0} ^{n-1} G_{n,j} (\cdot \mid \bar{\empd} _n) \times \bar{\empd} _n \right\} = \left\{ (\bar{G} _n \times \bar{\empd} _n)\right\},$$
of admissible control measures in $\M_1 \times \M_1$ is tight in the weak topology.
\end{lemma}
Lemma \ref{lemma:Step2} is a special case of Proposition 8.2.5 in \cite{Dupuis97} which establishes the result in the more general context of Markov chains. The proof is therefore omitted. 

Having established the tightness of the distributions of $\{ (\bar {G} _n, \bar {\empd} _n )\}$, the next step is to extend this to almost sure convergence of subsequences, in the weak topology, and study the limit.
\begin{lemma}
\label{lemma:Step3}
	Given any subsequence of $\{ (\bar{G} _n, \bar{\empd} _n )\}$, there exists a further subsequence that converges in distribution to some random variable $(\bar{G}, \bar{\empd})$, where $\bar{G} = \bar{\empd}$ a.s.
\end{lemma}

The result in Lemma \ref{lemma:Step3} is practically identical to part (b) of Lemma 2.5.1 in \cite{Dupuis97}. The only difference is that we must now appeal to the tightness proved in Lemma \ref{lemma:Step2}, whereas in \cite{Dupuis97} the underlying space is assumed to be compact; we omit the proof.

The next step is to show that the weak convergence of subsubsequences actually implies convergence of $\bar{G} _n$ and $\bar{\empd} _n$ in the $\tau$-topology. The result, Lemma \ref{lemma:Step4}, is a version of  Lemma 9.3.3 in \cite{Dupuis97} adapted to the case of independent and identically distributed random variables. Recall that we are already working with a specific subsubsequence (indexed by $n$) and on a probability space where the convergences $\bar{G}_n \weak \bar{\empd}$ and $\bar{\empd}_n \weak \bar{\empd}$ both occur with probability 1. Henceforth, $\M _1$ will be equipped with the $\tau$-topology.
\begin{lemma}
\label{lemma:Step4}
	Under the conditions of Lemmas \ref{lemma:Step1}-\ref{lemma:Step3}, there exists some subsequence of $n \in \N$ such that $\bar{G} _n \tauconv \bar{\empd}$ and $\bar{\empd} _n \tauconv \bar{\empd}$ w.p.\ 1 along this subsequence.
\end{lemma}

An important ingredient in the proof is the following inequality, which will appear frequently in what follows: For $a,b \geq0$ and $\sigma \geq 1$,
\begin{equation}
\label{eq:ineq}
 ab \leq e^{a \sigma} + \frac{1}{\sigma} (b \log (b) -b +1),
\end{equation}
see \cite{Dupuis97} for a proof.

With the almost sure convergence in the $\tau$-topology established the final results needed to prove Proposition \ref{prop:Upper} are obtained in Lemmas \ref{lemma:Step5} and \ref{lemma:Mapping}. We emphasize that this part of the proof is the main difference compared to Sanov's theorem and the essential ingredient needed to extend the standard result to Theorem \ref{thm:Laplace}. 
%The first result is that with probability 1 the limiting measure $\bar{\empd}$ is in the desired region of $\M _1$.
\begin{lemma}
\label{lemma:Step5}
Under the assumption $\int e^{\alpha wf} d \tilde{F} < \infty$ for any $\alpha > 0$, it holds that 
\begin{align*} 
\sup _n \Eb [\bar{G} _n (wf)] < \infty, \ \textrm{and} \ \sup _n \Eb [\bar{\empd} _n (wf)] < \infty.
\end{align*}
It follows that $\bar{\empd} \in \Gamma$ w.p.\ 1.
\end{lemma}

\begin{proof} Lemma \ref{lemma:Step1} shows that $\sup _n \Eb [\re (\bar{G} _n \mid \tilde{F})] < \infty$. Hence, each $\bar{G} _n$ has almost surely a well-defined Radon-Nikodym derivative $w_n$ with respect to $\tilde{F}$ and by definition
$$ \re (\bar{G} _n \mid \tilde{F}) = \int  w_n \log (w_n) d \tilde{F}. $$
Since $w$, $f$ and $w_n$ are all non-negative functions, the inequality (\ref{eq:ineq}) with $\sigma = 1$ gives
\begin{align*}
\bar{G} _n (wf) &= \int  wf d \bar{G} _n = \int  w f w_n d \tilde{F} \\
& \leq \int  e^{ w f}d\tilde{F} +  \int (w_n \log w_n - w_n -1 )d \tilde{F} \\
&= \int  e^{w f}d \tilde{F} + \int  \log w_n d \bar{G} _n \\
&= \int  e^{ w f}d \tilde{F} + \re(\bar{G} _n \mid \tilde{F}).
\end{align*}
Thus,
$$ \Eb [\bar{G} _n (wf)] \leq  \Eb[\re(\bar{G} _n \mid \tilde{F})] + \int e^{ w f}d \tilde{F},$$
and
$$ \sup _n \Eb [\bar{G} _n (wf)] \leq  \sup _n \Eb[\re(\bar{G} _n \mid \tilde{F})] + \int  e^{ w f}d \tilde{F}.$$
By the assumption and Lemma \ref{lemma:Step1} it holds that
$$ \sup _n \Eb [\bar{G} _n (wf)] < \infty. $$
Lemma \ref{lemma:Step4} proves that $\bar{G} _n \tauconv \bar{\empd}$ with probability 1. For $m \in \N$, $wf \wedge m$ is a bounded, measurable function and the $\tau$-convergence implies that
\begin{align*}
	\infty &> \limsup _n \Eb [\bar{G} _n (wf)] \geq \limsup _n \Eb [\bar{G} _n (wf \wedge m)] = \Eb [\bar{\empd} (wf \wedge m)].
\end{align*}
Taking $\limsup$ as $m \uparrow \infty$ together with a repeated use of Fatou's lemma shows that $\Eb [\bar{\empd} (wf)] < \infty$. Since $\empd (wf)$ is a non-negative random variable it follows that $\bar{\empd} \in \Gamma$ a.s.

The boundedness of $\sup _n \Eb [\bar{\empd} _n (wf)]$ is proved by repeatedly conditioning on the controlled process. 
\begin{align*}
	%\Eb[\bar{\empd} _n (wf)] &= \Eb [\frac{1}{n}\sum _{j=0} ^{n-1} w(\bar{X}_{n,j})f(\bar{X}_{n,j})] \\
	%&= \Eb \bigg[ \Eb [\frac{1}{n} w(\bar{X}_{n,n-1})f(\bar{X}_{n,n-1}) + \frac{1}{n}\sum _{j=0} ^{n-2} w(\bar{X}_{n,j})f(\bar{X}_{n,j}) \mid \bar{\empd} _{n,n-1}] \bigg] \\
	%&= \Eb \left [\frac{1}{n}\int w(x) f(x) G_{n,n-1}(dx \mid \bar{\empd} _{n,n-1}) + \bar{\empd} _{n,n-1} \right] \\
	%&= \Eb \bigg[ \frac{1}{n}G_{n,n-1}(wf \mid \bar{\empd} _{n,n-1}) + \Eb [\frac{1}{n} w(\bar{X}_{n,n-2})f(\bar{X}_{n,n-2}) \\
	%& \quad  + \frac{1}{n}\sum _{j=0} ^{n-3} w(\bar{X}_{n,j})f(\bar{X}_{n,j})\mid \bar{\empd} _{n,n-2}] \bigg] \\
	%&= \Eb [ \frac{1}{n}G_{n,n-1}(wf \mid \bar{\empd} _{n,n-1}) + \frac{1}{n}G_{n,n-2}(wf \mid \bar{\empd} _{n,n-2}) + \bar{\empd} _{n,n-2}].
	\Eb[\bar{\empd} _n (wf)] &= \Eb [\frac{1}{n}\sum _{j=0} ^{n-1} (wf)(\bar{X}_{n,j})] \\
	&= \Eb \bigg[ \Eb [\frac{1}{n} (wf)(\bar{X}_{n,n-1}) + \frac{1}{n}\sum _{j=0} ^{n-2} (wf)(\bar{X}_{n,j}) \mid \bar{\empd} _{n,n-1}] \bigg] \\
	&= \Eb \left [\frac{1}{n}\int (wf)(x) dG_{n,n-1}(x \mid \bar{\empd} _{n,n-1}) + \bar{\empd} _{n,n-1} (wf) \right] \\
	&= \Eb \bigg[ \frac{1}{n}G_{n,n-1}(wf \mid \bar{\empd} _{n,n-1})\\
	& \quad \qquad  + \Eb [\frac{1}{n} (wf)(\bar{X}(wf)_{n,n-2}) + \frac{1}{n}\sum _{j=0} ^{n-3} (wf)(\bar{X}_{n,j})\mid \bar{\empd} _{n,n-2}] \bigg] \\
	&= \Eb \biggl[ \frac{1}{n}G_{n,n-1}(wf \mid \bar{\empd} _{n,n-1}) + \frac{1}{n}G_{n,n-2}(wf \mid \bar{\empd} _{n,n-2}) + \bar{\empd} _{n,n-2}(wf)\biggr].
\end{align*}
Proceeding like this one obtains
\begin{align*}
	\Eb [\bar{\empd} _n (wf)] &= \Eb \biggl[\frac{1}{n}G_{n,0}(wf \mid \bar{\empd} _{n,0}) + ... +\frac{1}{n}G_{n,n-1}(wf \mid \bar{\empd} _{n,n-1})\biggr] \\
	& = \Eb \biggl[\frac{1}{n} \sum _{j=0} ^{n-1} \int w(x)f(x)dG _{n,j}(x \mid \bar{\empd} _{n,j})\biggr] \\
	&= \Eb [\bar{G} _n (wf)].
\end{align*}
This completes the proof.
\end{proof}

The final step for proving Proposition \ref{prop:Upper} is to show that the almost sure convergence in the $\tau$-topology on $\M _1$ implies almost sure convergence in the $\tau$-topology on $\M$ for the corresponding mapped measures $\Psi (\bar{G} _n ; \cdot)$ and $\Psi (\bar{\empd} _n ; \cdot)$.
\begin{lemma}
\label{lemma:Mapping}
	Along the subsequence for which the convergence in the $\tau$-topology holds, the convergences
	$$ \Psi (\bar{G} _n;\cdot) \tauconv \Psi (\bar{\empd};\cdot), \textrm{ and} \ \ \Psi (\bar{\empd} _n;\cdot) \tauconv \Psi (\bar{\empd};\cdot),$$
in $\M$ hold w.p. 1.
\end{lemma}
\begin{proof} Define for $G \in \M _1$ and $m \in \N $ a truncated version $\Psi _m (G;\cdot)$ of the mapping $\Psi$ as the finite measure given by
$$ \Psi _m (G;g) = \int (wf \wedge m) g dG,$$
for each bounded, measurable function $g$. The function $wf \wedge m$ is bounded and measurable and thus $\Psi _m$ is continuous with respect to the $\tau$-topology on $\M _1$. Therefore,
$$ \Psi _m (\bar{G} _n; g) = \int (wf \wedge m) g d\bar{G} _n \to \int (wf \wedge m)g d\bar{\empd} = \Psi _m (\bar{\empd};g),$$
as $n \to \infty$ along the particular subsequence. That is, for any bounded, measurable function $g$ it holds w.p.\ 1 that
\begin{align*} 
\lim _{n \to \infty} \Psi _m (\bar{G} _n ;g) = \Psi _m (\bar{\empd} ;g).
\end{align*}
Moreover, by Lemma \ref{lemma:Step5} the function $wf$ is integrable with respect to $\bar{\empd}$ w.p.1. Since $(wf \wedge m)g \to wfg$ as $m \to \infty$, the dominated convergence theorem implies that w.p.\ 1,
$$ \lim _{m \to \infty} \lim _{n \to \infty} \Psi _m (\bar{G} _n ; g) = \Psi (\bar{\empd};g),$$
for every bounded, measurable function $g$. Therefore, the desired convergence in $\M$ with the $\tau$-topology will follow if the order of the limit operators can be interchanged, which holds if 
\begin{equation}
\label{eq:op_change}
\lim _{m \to \infty} \sup _n \left | \int g wf d\bar{G}_n - \int g (wf \wedge m) d\bar{G}_n \right | = 0.
\end{equation}
The proof of Lemma \ref{lemma:Step4} involves an application of the Skorohod representation theorem, which results in the sequence $\{ \re (\bar{G} _n \mid \tilde{F} )\}$ being bounded a.s.\ and \eqref{eq:op_change} then follows by an application of the inequality \eqref{eq:ineq}; the argument is similar to Lemma 1.4.3(d) in \cite{Dupuis97}. 
%This proves the part concerning the admissible control sequence.

The aim is now to show that for each bounded measurable function $g$, $\Psi(\bar{\empd}  _n ; g) \to \Psi(\bar{\empd} ; g)$ in probability in a way such that for indicator functions $g=I_A$ we can appeal to the first Borel-Cantelli lemma to get almost sure convergence of the entire measure along some subsequence. For this, take any $\epsilon >0$ and consider the upper bound
\begin{align}
	&\bar{\Prob} \Big ( \,\Big| \int wfg d\bar{\empd} _n - \int wfg d\bar{G} _n \Big| \geq 3\epsilon \Big ) \notag \\
	&\leq \bar{\Prob} \Big ( \,\Big| \int (wf \wedge m)g d\bar{\empd} _n - \int wfg d\bar{\empd} _n \Big| \geq \epsilon \Big ) \label{eq:ineq1} \\
	& \quad + \bar{\Prob} \Big ( \,\Big| \int (wf\wedge m)g d\bar{\empd} _n - \int (wf \wedge m)g d\bar{G} _n \Big| \geq \epsilon \Big ) \label{eq:ineq2} \\ 
	& \quad + \bar{\Prob} \Big ( \,\Big| \int wfg d\bar{G} _n - \int (wf \wedge m)g d\bar{G} _n \Big| \geq \epsilon \Big ) , \label{eq:ineq3}
\end{align}
which holds for any $m \geq 0$. For $n \in \N$ and $j \in \{ 0,1,...,n-1\}$, define $\mathcal{\bar F}_{n,j}$ as the $\sigma$-algebra generated by the controlled process up to time $j$,
$$ \bar{\calF} _{n,j} = \sigma (\bar{\empd} _{n,0}, \bar{\empd} _{n,1},..., \bar{\empd} _{n,j}).$$
Recall that
$$ \bar{\Prob} (\bar{X} _{n,j} \in dy \mid \bar{\empd} _{n,0}, \bar{\empd} _{n,1},..., \bar{\empd} _{n,j}) = G_{n,j}(dy \mid \bar{\empd} _{n,j}),$$
that is $G_{n,j}(\cdot \mid \bar{\empd} _{n,j})$ is a regular conditional distribution for $\bar{X} _{n,j}$ given $\bar{\calF} _{n,j}$.
To take care of \eqref{eq:ineq1}, condition on the $\bar{\calF}_{n,j}$'s  to relate expectation of integrals with respect to $\bar{\empd} _{n,j}$ to integrals with respect to the control measures $G_{n,j}$,
\begin{align*}
	&\bar{\Prob} \Big (\, \Big| \int (wf \wedge m)g d\bar{\empd} _n - \int wfg d\bar{\empd} _n \Big| \geq \epsilon \Big ) \\
	& \quad \leq \frac{1}{\epsilon} \Eb \Big  [\, \Big| \int (wf \wedge m)g d\bar{\empd} _n - \int wfg d\bar{\empd} _n \Big| \Big  ] \\
	& \quad \leq \frac{||g||_{\infty}}{\epsilon} \Eb \Big  [\int (wf - wf \wedge m) d\bar{\empd} _n \Big  ] \\
	& \quad = \frac{||g||_{\infty}}{\epsilon} \Eb \Big  [\frac{1}{n} \sum _{j=0} ^{n-1} (wf - wf \wedge m) (\bar{X}_{n,j}) \Big ] \\
	& \quad = \frac{||g||_{\infty}}{\epsilon} \Eb\Big   [\frac{1}{n} \sum _{j=0} ^{n-1} \Eb \big [(wf - wf \wedge m) (\bar{X}_{n,j}) \mid \bar{\calF}_{n,j} \big ]\Big ] \\
	& \quad = \frac{||g||_{\infty}}{\epsilon} \Eb \Big [\frac{1}{n} \sum _{j=0} ^{n-1} \int (wf - wf \wedge m) dG_{n,j} \Big  ] \\
	& \quad = \frac{||g||_{\infty}}{\epsilon}\Eb \Big  [\int (wf - wf \wedge m) d\bar{G}_{n}\Big  ].
\end{align*}
Here, $||g||_{\infty}$ denotes the sup-norm over $\calX$. 

Next, consider the term in \eqref{eq:ineq3}. Analogous to the above,
\begin{align*}
	&\bar{\Prob} \Big (\ \Big| \int (wf \wedge m)g d\bar{G} _n - \int wfg d\bar{G} _n \ \Big| \geq \epsilon \Big )  \leq \frac{||g||_{\infty}}{\epsilon}\Eb \Big  [\int (wf - wf \wedge m) d\bar{G}_{n} \Big ].
\end{align*}
%Since $wf - wf\wedge m$ is non-negative and $\bar{G}_n$ w.p.\ 1 has a Radon-Nikodym derivative $w_n$ with respect to $\tilde{F}$, we can apply the inequality \eqref{eq:ineq}. For any $\sigma \geq 1$, 
%\begin{align*}
%	\Eb \Big  [\int (wf - wf \wedge m) d\bar{G}_{n} \Big ] \leq \int e^{\sigma (wf - wf\wedge m)}d\tilde{F} + \sup _n \frac{1}{\sigma}\Eb[\re (\bar{G}_n \mid \tilde{F})].
%\end{align*}
For any $\eta > 0$, define the set
\begin{align*}
	A(m,\eta) = \{ x\in \calX : (wf)(x) - (wf \wedge m)(x) > \eta \},
\end{align*}
and note that
\begin{align*}
	\int (wf - wf \wedge m ) d\bar{G} _n &= \int _{A(m,\eta)} (wf - wf \wedge m ) d\bar{G} _n \\
	& \quad + \int _{\calX \setminus A(m,\eta)} (wf - wf \wedge m ) d\bar{G} _n \\
	& \leq \int _{A(m,\eta)} (wf - wf \wedge m ) d\bar{G} _n + \eta.
\end{align*}
Together with the inequality \eqref{eq:ineq} this yields
\begin{align*}
\Eb \Big  [\int (wf - wf \wedge m) d\bar{G}_{n} \Big ] &\leq \int _{A(m,\eta)} e^{\sigma wf}d\tilde{F} + \sup _n \frac{1}{\sigma}\Eb[\re (\bar{G}_n \mid \tilde{F})] + \eta.
\end{align*}

%$$ \bar{\Prob} \Big ( \, \Big| \int (wf \wedge m)g d\bar{\empd} _n - \int (wf \wedge m)g d\bar{G} _n \Big| \geq \epsilon \Big ). $$
Finally, consider the term in \eqref{eq:ineq2}. To show that it converges to zero a martingale argument, similar to the proof of \cite[Lemma 8.2.7]{Dupuis97}, can be used. Since $G_{n,j} (\cdot \mid \bar{\empd} _{n,j})$ is a regular conditional distribution of $\bar{X} _{n,j}$, for any bounded measurable function $g: \calX \to \R$, w.p.1,
\begin{align*}
	&\Eb \biggl[(wf \wedge m)g (\bar{X} _{n,j}) - \int  (wf \wedge m) g d G _{n,j} (x\mid \bar{\empd} _{n,j}) \mid \bar{\calF} _{n,j}\biggr] \\
	& \quad = \Eb [(wf \wedge m)g(\bar{X} _{n,j}) \mid \bar{\calF} _{n,j}] - G_{n,j} ((wf \wedge m)g \mid \bar{\empd} _{n,j}) \\
	& \quad = G_{n,j} ((wf \wedge m)g \mid \bar {\empd} _{n,j}) - G_{n,j} ((wf \wedge m)g \mid \bar {\empd} _{n,j}) =0
\end{align*}
Hence, 
$$ \big \{ (wf \wedge m)g(\bar{X} _{n,j}) - G_{n,j} ((wf \wedge m)g \mid \bar {\empd} _{n,j}) \big \} _{j=0,1,...,n-1},$$
is a martingale difference sequence with respect to $\{ \bar{\calF} _{n,j} \}$. Moreover, for any $\epsilon > 0$,
\begin{align*}
& \bar{\Prob} \Big (\ \Big| \int (wf \wedge m)g d\bar{\empd} _n - \int (wf \wedge m)g d\bar{G} _n \ \Big| \geq \epsilon \Big ) \\
&\quad \leq \frac{1}{\epsilon ^2} \Eb \Big [ \ \Big| \int (wf \wedge m)g d\bar{\empd} _n - \int (wf \wedge m)g d\bar{G} _n \, \Big| ^2  \Big ] \\
	&\quad = \frac{1}{\epsilon ^2} \Eb \Big [ \frac{1}{n^2} \big ( \sum _{j=0} ^{n-1} ( (wf \wedge m)g (\bar{X} _{n,j}) - G _{n,j}((wf \wedge m)g \mid \bar{\empd} _{n,j}) ) \big ) ^2  \Big ] \\
	&\quad = \frac{1}{\epsilon ^2} \Eb \Big [ \frac{1}{n^2} \big ( \sum _{j=0} ^{n-1} ( (wf \wedge m)g (\bar{X} _{n,j}) - G _{n,j}((wf \wedge m)g \mid \bar{\empd} _{n,j}) ) ^2 \\ 
	& \quad \qquad \qquad + \sum _{i=1, j \neq i} ^{n-1} ( (wf \wedge m)g (\bar{X} _{n,i}) - G _{n,i}((wf \wedge m)g \mid \bar{\empd} _{n,i}) \\
	& \qquad \qquad \qquad \times ( (wf \wedge m)g (\bar{X} _{n,j}) - G _{n,j}((wf \wedge m)g \mid \bar{\empd} _{n,j}) \big ) \Big ].
\end{align*}
The second term vanishes when conditioning on the $\sigma$-algebras $\{ \bar{\calF} _{n,j} \}$. For the first term inside the expectation an upper bound on each of the summands is
\begin{align*}
	\Big ( (wf \wedge m)g (\bar{X} _{n,j}) - G _{n,j} ((wf \wedge m)g | \bar{\empd} _{n,j})\Big )^2 \leq 4 ||g|| ^2 _{\infty} m ^2.
\end{align*}
It follows that
\begin{align*}
	&\frac{1}{\epsilon ^2} \Eb \biggl [ \frac{1}{n^2} \big ( \sum _{j=0} ^{n-1} ( (wf \wedge m)g (\bar{X} _{n,j}) - G _{n,j}((wf \wedge m)g \mid \bar{\empd} _{n,j}) ) \big ) ^2 \biggr ] \\
	&\quad \leq \frac{4||g|| ^2 _{\infty} m^2}{\epsilon ^2 n}.	
\end{align*}

Combining the upper bounds for \eqref{eq:ineq1}-\eqref{eq:ineq3} yields the upper bound
\begin{align*}
 &\bar{\Prob} \Big ( | \Psi (\bar{\empd} _n ; wfg) - \Psi (\bar{G} _n ; wfg) | \geq 3\epsilon \Big ) \\
 & \leq \frac{2 ||g||_{\infty}}{\epsilon} \Big(\int_{A(m,\eta)} e^{\sigma wf }d\tilde{F}+ \frac{1}{\sigma}\sup _n \Eb[\re (\bar{G}_n \mid \tilde{F})] + \eta \Big)+ \frac{4||g||_\infty m ^2}{\epsilon ^2 n},
\end{align*}
which converges to $0$ if we send $n,m$ and $\sigma$ to $ \infty$ in that order. To see this, note that for any $\eta >0$, we can choose $m$ large enough so that $\int _{A(m,\eta)} e^{\sigma w f} d \tilde F$ is arbitrarily small. Indeed, we have the upper bound
\begin{align*}
	\tilde F(A(m,\eta)) \leq \frac{1}{\eta} \int _{A(m,\eta)} (wf - wf\wedge m) d \tilde F \leq \frac{1}{\eta} \int (wf - wf\wedge m) d \tilde F,
\end{align*}
and by the dominated convergence theorem, since $\tilde F(wf) < \infty$,
\begin{align*}
	\int (wf - wf \wedge m) d\tilde F \to 0, \ m \to \infty.
\end{align*}
Hence, for any choice of $\eta$, $\tilde F (A(m,\eta)) \to 0$ as $m \to \infty$. The convergence to $0$ for the upper bound now follows from the assumption $\int  e^{\sigma w f} d\tilde{F} < \infty$, and Lemma \ref{lemma:Step1}. 

To complete the proof we want to use the same argument as in Lemma 9.3.3 in \cite{Dupuis97} together with the convergence for $\Psi (\bar{G} _n; \cdot)$. For this it is not enough to have the above convergence in probability but one needs to be able to find a subsequence for which the convergence is fast enough for an application of the first Borel-Cantelli lemma. To illustrate, say that we want a subsequence $n_k$ such that for $n \geq n_k$, 
$$ \bar{\Prob} \Big ( | \Psi (\bar{\empd} _n ; wfg) - \Psi (\bar{G} _n ; wfg) | \geq 3\epsilon \Big ) \leq 2^{-k}.$$
Start by choosing $\eta \leq 2^{-k} \epsilon / 8$. Next, pick a $\sigma _k$ such that 
$$ \frac{1}{\sigma _k} \sup _n \Eb[\re (\bar{G}_n \mid \tilde{F})] \leq \frac{2^{-k} \epsilon}{8},$$
which is again possible due to Lemma \ref{lemma:Step1}. Having chosen this $\sigma_k$, pick $m_k$ sufficiently large so that
$$ \int _{A(m_k ,\eta)} e^{\sigma _kwf} d\tilde{F} \leq \frac{2^{-k} \epsilon}{8}.$$
Finally, pick $n_k$ such that
$$ \frac{m_k ^2}{n_{k}} \leq \frac{2^{-k} \epsilon ^2}{16}.$$
If $g$ is an indicator function, the above will yield a sequence $\{ n_k \}$ such that the probability of interest is smaller than $2^{-k}$. 
The key to the argument used in Lemma 9.3.3 in \cite{Dupuis97} is that, since the underlying space $\calX$ is assumed complete and separable, the Borel $\sigma$-algebra on $\calX$ is generated by a countable collection of Borel sets. Together with the above this yields the desired result, namely that
$$ \Psi (\bar{\empd} _n ; \cdot ) \tauconv \Psi (\empd ; \cdot),$$
in $\M$ w.p.\ 1. The reader is referred to \cite{Dupuis97} for the details.
\end{proof}
Lemmas \ref{lemma:Step1}-\ref{lemma:Mapping} complete the proof of Proposition \ref{prop:Upper}. 
\\

We end this section by proving that the function $I$ in \eqref{eq:Rate} has sequentially compact level sets in the $\tau$-topology.
% Thus Theorem \ref{thm:Laplace} is indeed a Laplace principle, for the weighted empirical measures on $\M$ equipped with the weak topology. 
\begin{proposition}
\label{prop:Rate}
	The function $I: \M \to [0,\infty]$ in \eqref{eq:Rate} has sequentially compact level sets on $\M$ equipped with the $\tau$-topology.
\end{proposition}
\begin{proof} Let
$$ C(K) = \{ \nu \in \M: I(\nu) \leq K \},$$
for $K < \infty$. Take any sequence $\{ \nu _n \} \subset C(K)$. Since $I(\nu _j) \leq K$ for each $j$, there exists a sequence $\{ G_n \} \subset \Delta \cap \Gamma$ such that $\Psi (G_j) = \nu _j$. Moreover, it must hold that $\re (G_j \mid \tilde{F}) \leq K + \epsilon$ for every $\epsilon > 0$ and each $j$. Hence,
$$ \sup _j \re (G_j \mid \tilde{F}) < \infty. $$
The relative entropy has compact level sets in the $\tau$-topology  \cite[Proposition 9.3.6]{Dupuis97}. Therefore, there exists some subsequence, also indexed by $n$,  
and some $G_*$ such that 
$$ G_{n} \tauconv G_* ,$$
and the corresponding (finite) measures $\nu _{n} =\Psi (G_{n})$, $\nu _* = \Psi (G _*)$ are in the set $C(K)$. It remains to prove that $\nu _{n} \tauconv \nu _*$. To this end, we note that by the same arguments as in Lemma \ref{lemma:Step5} it holds that 
$$ \sup _{n} G_{n} (wf) < \infty, $$
and
$$ G _* (wf) < \infty. $$
The conditions used to prove Lemma \ref{lemma:Mapping} are therefore satisfied and it follows in the same way that for every bounded measurable $g$,
$$ \nu _{n}(g) =\Psi (G_{n}; g) \to \Psi (G _* ; g) = \nu _* (g)$$
as $n \to \infty$. Hence, $\nu _{n} \tauconv \nu _*$ and the level set $C(K)$ is indeed sequentially compact in the $\tau$-topology. 
%This establishes that $I$ is a good rate function.
\end{proof}

\bibliographystyle{plain}
\bibliography{references}

\end{document}